\documentclass[a4paper,11pt]{article}

\usepackage{amsmath,amsfonts,amssymb,amsthm,amscd}
\usepackage{graphicx,mathtools}
\usepackage{perpage}
\usepackage{url}
\usepackage{color}
\usepackage[T1]{fontenc}
\usepackage{hyperref}
\usepackage[a4paper,scale={0.72,0.74},marginratio={1:1},footskip=7mm,headsep=10mm]{geometry}
\usepackage[latin9]{inputenc}
\usepackage{mathrsfs}
\usepackage{tikz}
\usepackage{caption}
\usepackage{subcaption}
\usepackage{float}
\usepackage{enumerate}
\usepackage[english]{babel}
\usepackage{hyperref}
\usepackage{setspace}
\usepackage{rotating}
\usepackage{bbm}
\makeatletter
\def\@captype{figure}
\makeatother

\numberwithin{equation}{section}

\newtheorem{theorem}{Theorem}[section]
\newtheorem{lemma}[theorem]{Lemma}
\newtheorem{proposition}[theorem]{Proposition}

\newtheorem{remark}[theorem]{Remark}
\newtheorem{definition}[theorem]{Definition}

\newtheorem{conjecture}[theorem]{Conjecture}

\newenvironment{hproof}{%
  \proof}{\endproof}

\newcommand{\eee}{\mathrm{e}}
\newcommand{\ddd}{\mathrm{d}}

\newcommand{\mi}{\mathrm{i}}

\title{Two-community noisy Kuramoto model}
\author{

J.\ M.\ Meylahn
\footnotemark[1]
}
\footnotetext[1]{
Mathematical Institute, Leiden University, P.O.\ Box 9512,
2300 RA Leiden, The Netherlands.
}
\begin{document}
\maketitle
\begin{abstract}
We study the noisy Kuramoto model for two interacting communities of oscillators, where we allow the interaction in and between communities to be positive or negative (but not zero). We find that, in the thermodynamic limit where the size of the two communities tends to infinity, this model exhibits non-symmetric synchronized solutions that bifurcate from the symmetric synchronized solution corresponding to the one-community noisy Kuramoto model, even in the case where the phase difference between the communities is zero and the interaction strengths are symmetric. The solutions are given by fixed points of a dynamical system. We find a critical condition for existence of a bifurcation line, as well as a pair of equations  determining the bifurcation line as a function of the interaction strengths. Using the latter we are able to classify the types of solutions that are possible and thereby identify the phase diagram of the system. We also analyze properties of the bifurcation line in the phase diagram and its derivatives, calculate the asymptotics, and analyze the synchronization level on the bifurcation line. Part of the proofs are numerically assisted. Lastly, we present some simulations illustrating the stability of the various solutions as well as the possible transitions between these solutions.

\medskip\noindent
{\it Key words and phrases.} Two-community network, phase oscillators, noisy Kuramoto model, McKean-Vlasov equations, phase diagram, bifurcation line.

\medskip\noindent
{\it Acknowledgment.} 
This work is supported by NWO Gravitation Grant 024.002.003--NETWORKS. The author is grateful to F.\ den Hollander for discussions and detailed comments, and to C. da Costa and E.\ Verbitskiy for fruitful input. 
\end{abstract}

\newpage

\section{Background and motivation}
The motivation for studying the two-community noisy Kuramoto model is two-fold. On the one hand, the suprachiasmatic nucleus (SCN) in the brain of mammals is responsible for biological time-keeping and consists of two communities of cells that exhibit synchronization \cite{W10}. On the other hand, there are recent studies of interacting particle systems with community structure, that reveal vast richness in behavior \cite{GC08, BR16, DP18, FC14}. The noisy Kuramoto model consists of a collection of oscillators with a mean-field interaction that favors alignment subject to external noisy \cite{SA88}.

The SCN is a cluster of neurons responsible for dictating the rhythm of bodily functions, most significantly the sleep-cycle. Malfunctioning of the SCN leads to a variety of health problems, ranging from epilepsy to narcolepsy. Remarkably, the network structure of the cluster is similar in all mammals, with the universal feature that it is split into two communities. In humans each cluster has a size of about $10^{4}$ neurons. It seems that this two-community structure is ideal, both for the robustness of the rhythm of the cluster not to be disturbed by unusual light inputs, as well as for the cluster to be adaptable enough to re-synchronize when there is a change in the light-dark cycle it is exposed to. As we will see below, this is reflected by the mathematical properties of the two-community noisy Kuramoto model, for which the interplay between positive and negative interactions introduces new features. The negative interaction, studied before in \cite{HS11}, \cite{HS12}, seems to play a key role in the appearance of a negative correlation between the neurons in the two communities in the SCN, resulting in new emergent behavior such as phase splitting \cite{PI08}. 

In the mathematics literature there have been recent studies on bipartite mean-field spin systems \cite{GC08}, as well as on the Ising block model \cite{BR16} and the asymmetric Curie-Weiss model \cite{DP18}, \cite{FC14}, where the splitting into two communities introduces interesting features, for example, the appearance of periodic orbits. These are discrete models which makes them hard to analyze. What makes the Kuramoto model considered here hard to analyze is that the interaction between phase oscillators in the Kuramoto model is \emph{non-linear}.

Also in \cite{SO15} the authors consider the two-community noisy Kuramoto model. They find an intricate phase diagram, with the system being able to take on a variety of different states. This confirms the observation that a simple modification in the network structure can greatly increase the complexity of the system. The results in \cite{SO15}, however, depend strongly on a Gaussian approximation for the phase distribution in each community (explained in \cite{SO13}), which allows for a reduction of the  dynamics to a low-dimensional setting. In this paper we do not rely on any such approximation.  

We have recently studied the noisy Kuramoto model on the \emph{hierarchical lattice} \cite{GA18}, finding conditions for synchronization either to propagate to all levels in the hierarchy or to vanish at a finite level. This analysis came about by writing down \emph{renormalized evolution equations} for the average phases in a block-community at a given hierarchical level in the hierarchical mean-field limit. In the present paper we allow for negative interactions across the communities, a situation we did not consider in the hierarchical model. 

In Section \ref{sec:basicproperties} we introduce the noisy Kuramoto model on the two-community network (see Fig.~\ref{fig:twocomnet}) and show that the empirical measures defined for each community evolve according to a McKean-Vlasov equation in the thermodynamic limit. We also give the steady-state solutions to these McKean-Vlasov equations and conjecture which values the phase difference between the average phases of the two communities can take in the steady state. In Section \ref{sec:critcon} we present results on the critical condition for synchronization in the case of symmetric interaction strengths and equal community sizes, first without disorder and then with disorder. By disorder we mean that the natural frequencies of the oscillators are taken from a distribution while without disorder means that all oscillators are assumed to have a natural frequency of zero. In Section \ref{sec:nonsym} we prove the conjecture from the previous section for a simplified version of the model where we take the interaction strengths to be symmetric and prove the existence of non-symmetric solutions in this case. Here symmetric solutions are solutions in which the synchronization level is the same in both communities while non-symmetric solutions are solutions where the synchronization level are non-zero and not the same in both communities. We also characterize the bifurcation line at which the non-symmetric solutions split off from the symmetric solutions, and expound a collection of results on the (asymptotic) properties of the bifurcation line in the phase diagram. Furthermore we analyze the synchronization level along the bifurcation line. Some of the proofs in Section \ref{sec:nonsym} are numerically assisted. Finally, in Section \ref{sec:simulation} we present some simulations illustrating the stability of the various solutions as well as the possible transitions between various steady-states.

\section{Basic properties}
\label{sec:basicproperties}
In Section \ref{sec:model} we define the model, in Section \ref{sec:mckean} we take the McKean-Vlasov limit, and in Section \ref{sec:stationarysol} we identify the stationary solutions.

\subsection{Model}
\label{sec:model}
We consider two communities of oscillators of size $N_{1}$ and $N_{2}$ with internal mean-field interactions of strength $\frac{K_{1}}{N_{1}}$ and $\frac{K_{2}}{N_{2}}$, respectively. In addition, the oscillators in community 1 experience a mean-field interaction with the oscillators in community 2 of strength $\frac{L_{1}}{N_{2}}$ and the oscillators in community 2 experience a mean-field interaction of strength $\frac{L_{2}}{N_{1}}$ with the oscillators in community 1. Here we will take $K_{1}, K_{2}\in \mathbb{R}$ to be positive and $L_{1}, L_{2}\in \mathbb{R}\setminus \{0\}$.

\begin{figure}
\centering
\includegraphics[scale=1.4]{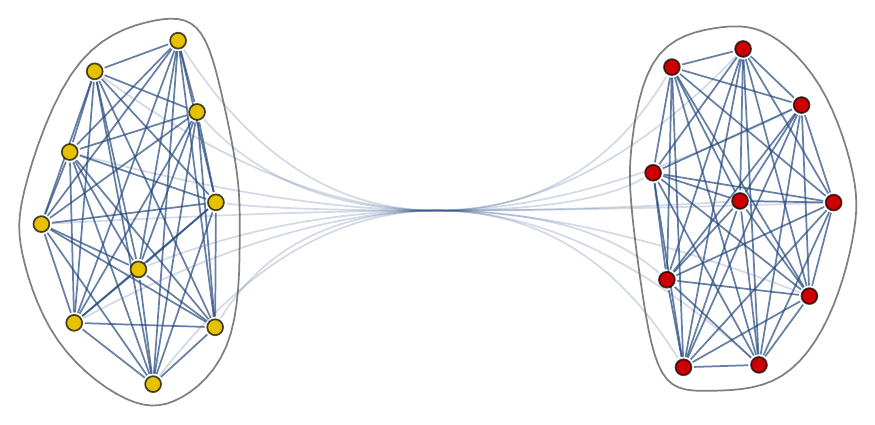}
\caption{Schematic picture of the two-community network, with community 1 consisting of $N_{1}$ yellow nodes and community 2 of $N_{2}$ red nodes. The interaction between yellow nodes has strength $K_{1}$, between red nodes strength $K_{2}$. Yellow nodes feel red nodes at strength $L_{1}$ and red nodes feel yellow nodes at strength $L_{2}$. Not all the interaction links between the communities are drawn.}
\label{fig:twocomnet}
\end{figure}

\begin{definition}[\textbf{Two-community noisy Kuramoto model}]
The phase angles of the oscillators in community 1 are denoted by $\theta_{1, i}$, $i=1, \cdots , N_{1}$, and their evolution on $\mathbb{S}=\mathbb{R}/2\pi$ is governed by the SDE 
\begin{eqnarray}
\label{eq:def2KM1}
\ddd \theta_{1, i}(t) =& \omega_{1, i}\ddd t + \frac{K_{1}}{N_{1}+N_{2}}\sum_{k=1}^{N_{1}}\sin(\theta_{1, k}(t) - \theta_{1, i}(t))\ddd t\nonumber\\
&+ \frac{L_{1}}{N_{1}+N_{2}}\sum_{l=1}^{N_{2}}\sin(\theta_{2, l}(t) - \theta_{1, i}(t))\ddd t + \sqrt{D}\ddd W_{1, i}(t).
\end{eqnarray}
The phase angles of the oscillators in community 2 are denoted by $\theta_{2, j}$,  $j=1, \cdots , N_{2}$, and their evolution on $\mathbb{S}=\mathbb{R}/2\pi$ is governed by the SDE 
\begin{eqnarray}
\label{eq:def2KM2}
\ddd \theta_{2, j}(t) =& \omega_{2, j}\ddd t + \frac{K_{2}}{N_{1}+N_{2}}\sum_{l=1}^{N_{2}}\sin(\theta_{2, l}(t) - \theta_{2, j}(t))\ddd t\nonumber\\
&+ \frac{L_{2}}{N_{1}+N_{2}}\sum_{k=1}^{N_{1}}\sin(\theta_{1, k}(t) - \theta_{2, j}(t))\ddd t + \sqrt{D}\ddd W_{2, j}(t).
\end{eqnarray}
Here, the natural frequencies $\omega_{1, i}$, $i=1, \ldots, N_{1}$, of the oscillators in community 1 are drawn independently from a probability distribution $\mu_{1}(\ddd \omega)$ on $\mathbb{R}$ and the natural frequencies $\omega_{2, i}$, $i=1, \ldots, N_{2}$, of the oscillators in community 2 are drawn independently from a probability distribution  $\mu_{2}(\ddd\omega)$ on $\mathbb{R}$, while $D>0$ is the noise strength, and $\big(W_{1, i}(t)\big)_{t\geq 0}, i=1, \ldots, N_{1}$, and $\big(W_{2, j}(t)\big)_{t\geq 0}, j=1, \ldots,N_{2}$, are independent  standard Brownian motions. For simplicity we take $\mu_{1}, \mu_{2}$ to be symmetric and have the same mean which we can assume to be zero without loss of generality.
\end{definition}

The model can alternatively be defined in terms of an interaction Hamiltonian and a weighted adjacency matrix, given by
\begin{align}
H_{N}(\theta_{1}, \ldots, \theta_{N}) = -\frac{1}{N}\sum_{i=1}^{N}\sum_{j=1}^{N}A_{i, j}\cos(\theta_{j}(t) - \theta_{i}(t)) + \sum_{i=1}^{N}\theta_{i}(t)\omega_{i}
\end{align}
with
\begin{align}
A := (A_{i, j})_{i, j=1, \ldots, N}= 
\begin{bmatrix}
    0 & K_{1}  & \dots  & K_{1} & L_{1} & L_{1} & \dots  & L_{1}  \\
    K_{1} & 0  & \dots  & K_{1} & L_{1} & L_{1}  & \dots  & L_{1}\\
    \vdots & \vdots  & \ddots & \vdots & L_{1} & L_{1} & \dots  & L_{1}\\
    K_{1} & K_{1} & \dots  & 0 & L_{1} & L_{1} & \dots  & L_{1}\\
    L_{2} & L_{2} & \dots  & L_{2} & 0 & K_{2} & \dots  & K_{2}\\
    L_{2} & L_{2} & \dots  & L_{2} & K_{2} & 0 & \dots  & K_{2}\\
    L_{2} & L_{2} & \dots  & L_{2} & \vdots & \vdots  & \ddots & \vdots \\
    L_{2} & L_{2} & \dots  & L_{2} & K_{2} & K_{2} & \dots  & 0
\end{bmatrix}
=
\begin{bmatrix}
    K_{1}\textbf{1}_{*} & L_{1}\textbf{1}\\
    L_{2}\textbf{1} & K_{2}\textbf{1}_{*}
\end{bmatrix},
\end{align}
where $\textbf{1}=$ all 1's and $\textbf{1}_{*}=$ all 1's, except for 0's on the diagonal. The model then reads 
\begin{align}
\ddd\theta_{i}(t) = \partial_{\theta_{i}}H_{N}(\theta_{1}, \ldots, \theta_{N})\ddd t +D\ddd W_{i}(t), \quad i=1, \ldots, N,
\end{align}
where $N=N_{1}+N_{2}$. Here, we identify phase angle $\theta_{i}$ with the oscillators in community 1 when $i\in [1, N_{1}]$ and with the oscillators in community 2 when $i\in(N_{1}, N_{1}+N_{2}]$. This representation of the model illustrates the network structure of the underlying interactions and in principle the adjacency matrix can be replaced by a matrix arising from a random graph model and has recently been addressed by a number of authors \cite{Bhamidi2018, Coppini2018, Delattre2016, Lucon2018, Oliveira2018}. This however significantly complicates the calculations since the interactions are no longer expressible in terms of a closed function of the empirical measure. The representation via the Hamiltonian may also provide a method for studying the stability properties of the stationary states.

The following \emph{order parameters} allow us to monitor the dynamics in each community:
\begin{eqnarray}
r_{1, N_{1}}(t) \eee^{\mi \psi_{1, N_{1}}(t)} =& \frac{1}{N_{1}}\sum_{k =1}^{N_{1}}\eee^{\mi\theta_{1, k}(t)},\\
r_{2, N_{2}}(t) \eee^{\mi \psi_{2, N_{2}}(t)} =& \frac{1}{N_{2}}\sum_{l =1}^{N_{2}}\eee^{\mi\theta_{2, l}(t)},
\end{eqnarray}
where $r_{1, N_{1}}(t)\in [0, 1]$ and $r_{2, N_{2}}(t)\in [0, 1]$ represent the \emph{synchronization levels}, and $\psi_{1, N_{1}}(t)$ and $\psi_{2, N_{2}}(t)$ represent the \emph{average phases}, in community 1 and 2, respectively. Using these order parameters, we can rewrite the evolution equations in \eqref{eq:def2KM1} and \eqref{eq:def2KM2} as
\begin{eqnarray}
\label{eq:rewritenevolution1}
\ddd \theta_{1, i}(t) =& \omega_{1, i}\ddd t + \frac{K_{1}N_{1}}{N_{1}+N_{2}}r_{1, N_{1}}(t)\sin(\psi_{1, N_{1}}(t) - \theta_{1, i}(t))\ddd t\nonumber\\
&+ \frac{L_{1}N_{2}}{N_{1}+N_{2}}r_{2, N_{2}}(t)\sin(\psi_{2, N_{2}}(t) - \theta_{1, i}(t))\ddd t + \sqrt{D}\ddd W_{1, i}(t)
\end{eqnarray}
and
\begin{eqnarray}
\label{eq:rewritenevolution2}
\ddd \theta_{2, j}(t) =& \omega_{2, j}\ddd t + \frac{K_{2}N_{2}}{N_{1}+N_{2}}r_{2, N_{2}}(t)\sin(\psi_{2, N_{2}}(t) - \theta_{2, j}(t))\ddd t\nonumber\\
&+ \frac{L_{2}N_{1}}{N_{1}+N_{2}}r_{1, N_{1}}(t)\sin(\psi_{1, N_{1}}(t)- \theta_{2, j}(t))\ddd t + \sqrt{D}\ddd W_{2, j}(t).
\end{eqnarray}

\subsection{McKean-Vlasov limit}
\label{sec:mckean}

We assume that the sizes of the communities are related to one another by setting $N_{1} = \alpha_{1}N$ and $N_{2}=\alpha_{2}N$, $\alpha_{1} + \alpha_{2} = 1$. In the limit as $N\to\infty$, we expect the angle density of oscillators in each community to follow a \emph{McKean-Vlasov equation}. Define the empirical measure for each community $(\theta\in\mathbb{S}, \omega\in\mathbb{R})$:
\begin{align}
\nu_{N_{1}, t}(\ddd\theta, \ddd\omega)&:=\frac{1}{N_{1}}\sum_{i=1}^{N_{1}}\delta_{(\theta_{1, i}(t), \omega_{1,i})}(\ddd\theta, \ddd\omega),\\
\nu_{N_{2}, t}(\ddd\theta, \ddd\omega)&:=\frac{1}{N_{2}}\sum_{j=1}^{N_{2}}\delta_{(\theta_{2, j}(t), \omega_{2,j})}(\ddd\theta, \ddd\omega).
\end{align}
\begin{proposition}[\textbf{McKean-Vlasov limit}]
\label{prop:McKeanlimit}
In the limit as $N\to \infty$, the empirical measure $\nu_{N_{1}, t}(\ddd\theta, \ddd\omega)$ converges to $\nu_{1, t}(\ddd\theta, \ddd\omega) = p_{1}(t;\theta,\omega)\,\ddd\theta\,\ddd\omega$, and the empirical measure $\nu_{N_{2}, t}(\ddd\theta, \ddd\omega)$ converges to $\nu_{2, t}(\ddd\theta, \ddd\omega) = p_{2}(t;\theta,\omega)\,\ddd\theta\,\ddd\omega$, where $p_{1}(t;, \theta, \omega)$ evolves according to
\begin{eqnarray}
\frac{\partial p_{1}(t; \theta, \omega)}{\partial t} = \frac{D}{2}\frac{\partial^{2}p_{1}(t; \theta, \omega)}{\partial \theta^{2}} - \frac{\partial}{\partial \theta} \big[v_{1}(t;\theta, \omega)p_{1}(t; \theta, \omega)\big]
\end{eqnarray}
with
\begin{eqnarray}
v_{1}(t;\theta, \omega) = \omega + \alpha_{1}K_{1}r_{1}(t)\sin(\psi_{1}(t) - \theta) + \alpha_{2}L_{1}r_{2}(t)\sin(\psi_{2}(t) - \theta),
\end{eqnarray}
and $p_{2}(t;\theta, \omega)$ evolves according to
\begin{eqnarray}
\frac{\partial p_{2}(t; \theta, \omega)}{\partial t} = \frac{D}{2}\frac{\partial^{2}p_{2}(t; \theta, \omega)}{\partial \theta^{2}} - \frac{\partial}{\partial \theta} \big[v_{2}(t;\theta, \omega)p_{2}(t; \theta, \omega)\big]
\end{eqnarray}
with
\begin{eqnarray}
v_{2}(t;\theta, \omega) = \omega + \alpha_{2}K_{2}r_{2}(t)\sin(\psi_{2}(t) - \theta) + \alpha_{1}L_{2}r_{1}(t)\sin(\psi_{1}(t) - \theta).
\end{eqnarray}
Here, $r_{1}(t), r_{2}(t), \psi_{1}(t)$ and $\psi_{2}(t)$ are defined by 
\begin{align}
r_{1}(t)\eee^{\mi \psi_{1}(t)}:=&\int_{\mathbb{S}\times\mathbb{R}}\nu_{1, t}(\ddd\theta, \ddd\omega)\,\eee^{\mi \theta},\\
r_{2}(t)\eee^{\mi \psi_{2}(t)}:=&\int_{\mathbb{S}\times\mathbb{R}}\nu_{2, t}(\ddd\theta, \ddd\omega)\,\eee^{\mi \theta}.
\end{align}
The convergence is in $\mathcal{C}([0, T], \mathcal{M}_{1}(\mathbb{S}\times \mathbb{R}))$ and takes place for any $T>0$. Here we consider annealed convergence with respect to the natural frequencies.
\end{proposition}
\begin{proof}
The proof is analogous to that in the case of the one-community noisy Kuramoto model in \cite{DdH96} with straightforward modifications.
\end{proof}

\subsection{Stationary solutions}
\label{sec:stationarysol}
The stationary solutions of the McKean-Vlasov limit in Proposition \ref{prop:McKeanlimit} give the possible states the system can assume in the long time limit. These are presented in the next proposition. 

\begin{proposition}[\textbf{Stationary solutions}]
In the cases $r_{1}=r_{2}=0$ and $r_{1}, r_{2}>0$, the stationary density $p_{1}(\theta, \omega)$ solves the equation
\begin{eqnarray}
\label{eq:stationarypde1}
0 = \frac{D}{2}\frac{\partial^{2}p_{1}(\theta, \omega)}{\partial \theta^{2}} - \frac{\partial}{\partial \theta} \big[v_{1}(\theta, \omega)p_{1}(\theta, \omega)\big],
\end{eqnarray}
which has solution

\begin{align}
\label{eq:p1solution}
p_{1}(\theta, \omega) =& \frac{A_{1}(\theta, \omega)}{\int_{\mathbb{S}} \,\ddd\phi A_{1}(\phi, \omega)},
\end{align}
where
\begin{align}
A_{1}(\theta, \omega) = B_{1}(\theta, \omega)\Big(\eee^{\frac{4\pi\omega}{D}}\int_{\mathbb{S}}\frac{\ddd\phi}{B_{1}(\phi, \omega)} + (1-\eee^{\frac{4\pi\omega}{D}})\int_{0}^{\theta}\frac{\ddd\phi}{B_{1}(\phi, \omega)}\Big)
\end{align}
with
\begin{align}
B_{1}(\theta, \omega) = \exp\Big[ \frac{2\omega\theta}{D} + \frac{2\alpha_{2}L_{1}r_{2}\cos(\psi_{2}-\theta)}{D} + \frac{2\alpha_{1}K_{1}r_{1}\cos(\psi_{1} - \theta)}{D}\Big].
\end{align}
The stationary density $p_{2}(\theta, \omega)$, solves the equation
\begin{eqnarray}
\label{eq:stationarypde2}
0 = \frac{D}{2}\frac{\partial^{2}p_{2}(\theta, \omega)}{\partial \theta^{2}} - \frac{\partial}{\partial \theta} \big[v_{2}(\theta, \omega)p_{2}(\theta, \omega)\big],
\end{eqnarray}
which has solution
\begin{align}
\label{eq:p2solution}
p_{2}(\theta, \omega) =& \frac{A_{2}(\theta, \omega)}{\int_{\mathbb{S}} \,\ddd\phi A_{2}(\phi, \omega)},
\end{align}
where
\begin{align}
A_{2}(\theta, \omega) = B_{2}(\theta, \omega)\Big(\eee^{\frac{4\pi\omega}{D}}\int_{\mathbb{S}}\frac{\ddd\phi}{B_{2}(\phi, \omega)} + (1-\eee^{\frac{4\pi\omega}{D}})\int_{0}^{\theta}\frac{\ddd\phi}{B_{2}(\phi, \omega)}\Big)
\end{align}
with
\begin{align}
B_{2}(\theta, \omega) = \exp\Big[ \frac{2\omega\theta}{D} + \frac{2\alpha_{1}L_{2}r_{1}\cos(\psi_{1}-\theta)}{D} + \frac{2\alpha_{2}K_{2}r_{2}\cos(\psi_{2} - \theta)}{D}\Big].
\end{align}
In addition, the following self-consistency equations must be satisfied:
\begin{align}
\label{eq:thmselfcons}
r_{1} &= V_{1}^{\mu_{1}}(r_{1}, r_{2}) := \int_{\mathbb{R}}\mu_{1}(\ddd\omega)\int_{\mathbb{S}}\,\,\ddd\theta\cos(\psi_{1}-\theta)\,p_{1}(\theta, \omega),\\
r_{2} &= V_{2}^{\mu_{2}}(r_{1}, r_{2}) :=\int_{\mathbb{R}}\,\mu_{2}(\ddd\omega)\int_{\mathbb{S}}\,\ddd\theta\cos(\psi_{2}-\theta)\,p_{2}(\theta, \omega),\nonumber\\
0 &= U_{1}^{\mu_{1}}(r_{1}, r_{2}) := \int_{\mathbb{R}}\,\mu_{1}(\ddd\omega)\int_{\mathbb{S}}\,\ddd\theta\sin(\psi_{1}-\theta)\,p_{1}(\theta, \omega),\nonumber\\
0 &= U_{2}^{\mu_{2}}(r_{1}, r_{2}) := \int_{\mathbb{R}}\,\mu_{2}(\ddd\omega)\int_{\mathbb{S}}\,\ddd\theta\sin(\psi_{2}-\theta)\,p_{2}(\theta, \omega).\nonumber
\end{align}
\end{proposition}
\begin{proof}
Note that in the case when $r_{1}=r_{2}=0$, both stationary densities are uniform on $\mathbb{S}$, i.e., $p_{1}(\theta, \omega) = p_{2}(\theta, \omega)=\frac{1}{2\pi}$, which satisfies \eqref{eq:stationarypde1} and \eqref{eq:stationarypde2}. The proof in the case when $r_{1}, r_{2}>0$ is analogous to the calculation given in \cite[Solution to Exercise X.33]{H00}.
\end{proof}

\begin{remark}
In the simplified version of the model we will consider below, we are able to prove that solutions of the type $r_{1}=0$ and $r_{2}>0$ (or vice versa) are not possible, but it is difficult to prove this in the general case considered above.  
\end{remark}

In order to understand the steady-state phase difference between the communities, we proceed heuristically as follows. For the stationary solutions we assume that $\hspace{0.15cm}r_{1}(t),\hspace{0.15cm} r_{2}(t), \\ \psi_{1}(t),\psi_{2}(t)$ reach their steady-state values $r_{1}, r_{2}, \psi_{1}, \psi_{2}$ as $t\rightarrow\infty$ and assume that the parameters of the system are such that $r_{1}, r_{2} >0$. For the synchronization levels the possible steady-state values are computed by solving the self-consistency equations in \eqref{eq:thmselfcons}. For the average phases we use standard It\^{o}-calculus to compute their evolution
\begin{align}
\label{eq:Ito}
\ddd \psi_{m}(t) = \sum_{j=1}^{N_{m}}\frac{\partial \psi_{m}}{\partial \theta_{m, j}} \ddd\theta_{m, j} +\frac{1}{2}\sum_{j=1}^{N_{m}}\frac{\partial^{2}\psi_{m}}{\partial \theta_{m, j}^{2}}(\ddd\theta_{m, j})^{2},\quad m\in\{1, 2\}.
\end{align}
From the definition of the order parameters we have
\begin{align}
\label{eq:Ito1}
\frac{\partial \psi_{m}}{\partial \theta_{m, j}} = \frac{1}{N_{m}r_{m}(t)}\cos(\psi_{m}(t)-\theta_{m, j}(t)),\quad m\in\{1, 2\},
\end{align}
and 
\begin{align}
\label{eq:Ito2}
\frac{\partial^{2}\psi_{m}}{\partial \theta_{m, j}^{2}} =& \frac{1}{N_{m} r_{m}(t)}\sin(\psi_{m}(t) - \theta_{m, j}(t))\\ 
&- \frac{2}{(N_{m}r_{m}(t))^{2}}\sin(\psi_{m}(t)-\theta_{m, j}(t))\cos(\psi_{m}(t)-\theta_{m, j}(t)),\quad m\in\{1, 2\}\nonumber.
\end{align}
Substituting \eqref{eq:Ito1}--\eqref{eq:Ito2} and \eqref{eq:rewritenevolution1}--\eqref{eq:rewritenevolution2} into \eqref{eq:Ito}, setting $N_{m}=\alpha_{m}N$ and taking the large-$N$ limit, we get the equations
\begin{align}
\label{eq:psi1Ito}
\ddd \psi_{1}(t) =& \Bigg(\frac{K_{1}\alpha_{1}}{2}\int_{\mathbb{S}}\ddd\theta \int_{\mathbb{R}}\mu_{1}(\ddd\omega)\,\cos(\psi_{1}(t)-\theta)\sin(\psi_{1}(t)-\theta)p_{1}(t;\theta, \omega)\\
&+\frac{L_{1}\alpha_{2}r_{2}(t)}{2r_{1}(t)}\int_{\mathbb{S}}\ddd\theta \int_{\mathbb{R}}\mu_{1}(\ddd\omega)\,\cos(\psi_{1}(t)-\theta)\sin(\psi_{2}(t)-\theta)p_{1}(t;\theta, \omega)\nonumber\\
&+ \frac{1}{r_{1}(t)}\int_{\mathbb{S}}\ddd\theta \int_{\mathbb{R}} \mu_{1}(\ddd\omega)\,\omega \cos(\psi_{1}(t) - \theta)p_{1}(t;\theta, \omega)\nonumber\\
&+\frac{D}{2}\int_{\mathbb{S}}\ddd\theta \int_{\mathbb{R}} \mu_{1}(\ddd\omega)\,\sin(\psi_{1}(t) - \theta)p_{1}(t;\theta, \omega)\Bigg)\ddd t ,\nonumber
\end{align}
\begin{align}
\label{eq:psi2Ito}
\ddd \psi_{2}(t) =& \Bigg(\frac{K_{2}\alpha_{2}}{2}\int_{\mathbb{S}}\ddd\theta \int_{\mathbb{R}} \mu_{2}(\ddd\omega)\,\cos(\psi_{2}(t)-\theta)\sin(\psi_{2}(t)-\theta)p_{2}(t;\theta, \omega)\\
&+\frac{L_{2}\alpha_{1}r_{1}(t)}{2r_{2}(t)}\int_{\mathbb{S}}\ddd\theta \int_{\mathbb{R}} \mu_{2}(\ddd\omega)\,\cos(\psi_{2}(t)-\theta)\sin(\psi_{1}(t)-\theta)p_{2}(t;\theta, \omega)\nonumber\\
&+ \frac{1}{r_{2}(t)}\int_{\mathbb{S}}\ddd\theta \int_{\mathbb{R}} \mu_{2}(\ddd\omega)\,\omega \cos(\psi_{2}(t) - \theta)p_{2}(t;\theta, \omega)\Bigg)\ddd t .\nonumber\\
&+\frac{D}{2}\int_{\mathbb{S}}\ddd\theta \int_{\mathbb{R}} \mu_{2}(\ddd\omega)\,\sin(\psi_{2}(t) - \theta)p_{2}(t;\theta, \omega)\Bigg)\ddd t .\nonumber
\end{align}
Due to the last two self-consistency equations in \eqref{eq:thmselfcons} the last line of \eqref{eq:psi1Ito} and \eqref{eq:psi2Ito} is zero. For the steady-state average phases in the case when $\mu_{1}=\mu_{2}=\delta_{0}$, we must therefore simultaneously solve the equations 
\begin{align}
0 =& \frac{K_{1}\alpha_{1}}{2}\int_{\mathbb{S}}\cos(\psi_{1}-\theta)\sin(\psi_{1}-\theta)p_{1}(\theta, 0)\ddd\theta\\
&+\frac{L_{1}\alpha_{2}r_{2}}{2r_{1}}\int_{\mathbb{S}}\cos(\psi_{1}-\theta)\sin(\psi_{2}-\theta)p_{1}(\theta, 0)\ddd\theta\nonumber,
\end{align}
\begin{align}
0 =& \frac{K_{2}\alpha_{2}}{2}\int_{\mathbb{S}}\cos(\psi_{2}-\theta)\sin(\psi_{2}-\theta)p_{2}(\theta, 0)\ddd\theta\\
&+\frac{L_{2}\alpha_{1}r_{1}}{2r_{2}}\int_{\mathbb{S}}\cos(\psi_{2}-\theta)\sin(\psi_{1}-\theta)p_{2}(\theta, 0)\ddd\theta\nonumber.
\end{align}
Since the system is invariant under rotations, we can set one of the two angles to zero. If we set $\psi_{1}=0$, then we see that the equation for $\psi_{2}$ is satisfied by taking $\psi_{2}=0$ or $\psi_{2}=\pi$. The above calculation is not rigorous, but does suggest the following conjecture.

\begin{conjecture}[\textbf{Steady-state phase difference}]
\label{con:phasedifference}
In the system without disorder, the phase difference $\psi=\psi_{2}-\psi_{1}$ between the two communities in the two-community noisy Kuramoto model with $K_{1}=K_{2}=K$ and $L_{1}=L_{2}=L\neq 0$ in the steady state can only be $\psi=0$ or $\psi=\pi$.
\end{conjecture}
The intuition for this conjecture is that the system will try to maximize the interaction strength between oscillators in order to achieve the highest synchronization in each community. This will be achieved at $\psi=0$ when $L>0$ and at $\psi=\pi$ when $L<0$. The other combinations ($\psi=0$ with $L<0$ and $\psi=\pi$ with $L>0$) should also be possible, but should not be stable. For an illustration of stability properties obtained via simulations, we refer the reader to Section~\ref{sec:simulation}.  
 
\section{Symmetric interaction with fixed phase difference}
\label{sec:critcon}
In this section we pick $L_{1}=L_{2} = L$, $K_{1}=K_{2} = K$,  $\alpha_{1} =\alpha_{2}$, $D=1$. In Section \ref{sec:withoutdisorder} we consider the case where the natural frequency of the oscillators is zero, and in Section \ref{sec:withdisorder} the case where the natural frequency of the oscillators is drawn from a symmetric distribution $\mu$ on $\mathbb{R}$.
\subsection{Without disorder}
\label{sec:withoutdisorder}
Here we take $\mu_{1} = \mu_{2} = \delta_{0}$. This simplifies \eqref{eq:p1solution} and \eqref{eq:p2solution} to 
\begin{align}
p_{1}(\theta) &= \frac{\exp\Big[ Lr_{2}\cos(\psi_{2} - \theta) + Kr_{1}\cos(\psi_{1} - \theta)\Big]}{\int_{\mathbb{S}}\ddd\phi\,\exp\Big[ Lr_{2}\cos(\psi_{2} - \phi) + Kr_{1}\cos(\psi_{1} - \phi)\Big]},
\label{eq:simplep1}\\
p_{2}(\theta) &= \frac{\exp\Big[ Lr_{1}\cos(\psi_{1} - \theta) + Kr_{2}\cos(\psi_{2} - \theta)\Big]}{\int_{\mathbb{S}}\ddd\phi\,\exp\Big[ Lr_{1}\cos(\psi_{1} - \phi) + Kr_{2}\cos(\psi_{2} - \phi)\Big]}.
\label{eq:simplep2}
\end{align}

The self-consistency equations for $r_{1}$ and $r_{2}$ in \eqref{eq:thmselfcons} can be written in the form
\begin{align}
\label{eq:selfconsW}
r_{1} &= \frac{(a_{1}\cos\psi_{1}+b_{1}\sin\psi_{1})}{2}W\Big(\sqrt{a_{1}^{2}+b_{1}^{2}}\;\Big),\\
r_{2} &= \frac{(a_{2}\cos\psi_{2}+b_{2}\sin\psi_{2})}{2}W\Big(\sqrt{a_{2}^{2}+b_{2}^{2}}\;\Big),\nonumber
\end{align}
where $W(x) = \frac{2V(x)}{x}, x\in (0, \infty),$ with
\begin{align}
\label{eq:Vdefinition}
V(x) = \frac{\int_{\mathbb{S}}\,\ddd\theta\,\cos\theta \,\eee^{x\cos\theta}}{\int_{\mathbb{S}}\,\ddd\theta\,\eee^{x\cos\theta}},\quad x\in[0, \infty).
\end{align}
The definitions of $a_{1}, a_{2}, b_{1}$ and $b_{2}$ will be given below. The function $V(x)$ is the same function that appears in the self-consistency equation of the one-community noisy Kuramoto model \cite[Equation 2.2]{Giacomin2014}. To see why the self-consistency equations can be written as in \eqref{eq:selfconsW}, note that 
\begin{align}
\int_{\mathbb{S}}\ddd\theta\,\eee^{a\cos\theta + b\sin\theta} = 2\pi \text{I}_{0}(\sqrt{a^{2}+b^{2}}),
\end{align}
with $\text{I}_{m}(x):= \frac{1}{2\pi}\int_{\mathbb{S}}\ddd\theta (\cos\theta)^{m}\exp(x\cos\theta)$ the modified Bessel functions of the first kind, so that 
\begin{align}
\label{eq:cosint}
\int_{\mathbb{S}}\ddd\theta\,\cos\theta\,\eee^{a\cos\theta + b\sin\theta} &= \frac{\partial}{\partial a}2\pi\text{I}_{0}(\sqrt{a^{2}+b^{2}})
= \frac{2\pi a \text{I}_{1}(\sqrt{a^{2}+b^{2}})}{\sqrt{a^{2}+b^{2}}},\\
\int_{\mathbb{S}}\ddd\theta\,\sin\theta\,\eee^{a\cos\theta + b\sin\theta} &= \frac{\partial}{\partial b}2\pi\text{I}_{0}(\sqrt{a^{2}+b^{2}})= \frac{2\pi b \text{I}_{1}(\sqrt{a^{2}+b^{2}})}{\sqrt{a^{2}+b^{2}}}.\nonumber
\end{align} 
Here we have used the identity $I_{0}(x)=I_{1}(x)$ given in \cite[9.6.27]{Abramowitz65}.
Using \eqref{eq:cosint} and the trigonometric identity $\cos(a-b) = \cos a\cos b + \sin a\sin b$, $a, b \in \mathbb{R}$, we can rewrite the self-consistency equations for $r_{1}$ and $r_{2}$ as
\begin{align}
r_{1} &= \frac{(a_{1}\cos\psi_{1}+b_{1}\sin\psi_{1}) \text{I}_{1}(\sqrt{a_{1}^{2}+b_{1}^{2}})}{\sqrt{a_{1}^{2}+b_{1}^{2}}\;\text{I}_{0}(\sqrt{a_{1}^{2}+b_{1}^{2}})},\\
r_{2} &= \frac{(a_{2}\cos\psi_{2}+b_{2}\sin\psi_{2}) \text{I}_{1}(\sqrt{a_{2}^{2}+b_{2}^{2}})}{\sqrt{a_{1}^{2}+b_{1}^{2}}\;\text{I}_{0}(\sqrt{a_{2}^{2}+b_{2}^{2}})}\nonumber,
\end{align}
where
\begin{align}
a_{1} &= K r_{1}\cos\psi_{1} + Lr_{2}\cos\psi_{2},\quad b_{1} = Kr_{1}\sin\psi_{1} + Lr_{2}\sin\psi_{2},\\
\label{eq:generalselfcons}
a_{2} &= K r_{2}\cos\psi_{2} + Lr_{1}\cos\psi_{1},\quad b_{2} =  K r_{2}\sin\psi_{2} + Lr_{1}\sin\psi_{1}.\nonumber
\end{align}
Note that
\begin{align}
&a_{1}^{2} + b_{1}^{2} = K^{2}r_{1}^{2} + L^{2}r_{2}^{2} + 2KLr_{1}r_{2}\cos\psi,\\
&a_{2}^{2} + b_{2}^{2} = K^{2}r_{2}^{2} + L^{2}r_{1}^{2} + 2KLr_{1}r_{2}\cos\psi,
\end{align}
where we recall $\psi = \psi_{2}-\psi_{1}$.
The most suggestive form of the self-consistency equations is in terms of $K,L$ and the phase difference $\psi$:
\begin{align}
r_{1} &= \frac{(Kr_{1} + Lr_{2}\cos\psi)}{2} W\Big(\sqrt{K^{2}r_{1}^{2} + L^{2}r_{2}^{2} + 2KLr_{1}r_{2}\cos\psi}\,\Big),\nonumber\\
r_{2} &= \frac{(Kr_{2} + Lr_{1}\cos\psi)}{2} W\Big(\sqrt{K^{2}r_{2}^{2} + L^{2}r_{1}^{2} + 2KLr_{1}r_{2}\cos\psi}\,\Big)\label{eq:selfconsfinal}
\end{align}
and is obtained by substituting the expressions for $a_{1}, a_{2}, b_{1}$ and $b_{2}$ into \eqref{eq:selfconsW}.

\begin{proposition}[\textbf{Properties of $V$}]\ \\
\label{prop:vproperties}
\vspace{-0.5cm}
\begin{enumerate}
\item $V(0) = 0$.
\item $V'(0) = \frac{1}{2}$.
\item $x\mapsto V(x)$ is strictly increasing on $[0, \infty)$.
\item $x\mapsto V(x)$ is strictly concave on $[0, \infty)$.
\item $V(x) <\frac{x}{2}$ for $x\in (0, \infty)$.
\item $\lim_{x\to\infty} V(x) = 1$.
\item $V(-x) = -V(x)$ for all $x\in (0, \infty)$.
\end{enumerate}
\end{proposition}
\begin{proof}
Properties 1, 2, 3 and 6 are easily verified. Property 4 is proven by applying Lemma 4 in \cite{PP81} (see Appendix \ref{app:concavity} for a comprehensive proof). Property 5 is a direct consequence of properties 1, 2 and 4. For Property 7, use $-\cos(\theta) = \cos(\pi-\theta)$ to write
\begin{align}
V(-x) = \frac{\int_{\mathbb{S}}\ddd\theta\,\cos\theta\eee^{x\cos(\pi-\theta)}}{\int_{\mathbb{S}}\ddd\theta\,\eee^{x\cos(\pi-\theta)}}.
\end{align} 
By performing the change of variable $\phi = \pi - \theta$, we get $V(-x) = -V(x)$.
\end{proof}

\begin{proposition}[\textbf{Properties of $W$}]\ \\
\label{prop:wproperties}
\vspace{-0.5cm}
\begin{enumerate}
\item $\lim_{x\downarrow 0}W(x) = 1$.
\item $x\mapsto W(x)$ is continuous and strictly decreasing on $[0, \infty)$.
\item $\lim_{x\to\infty}W(x) = 0$.
\end{enumerate}
\end{proposition}
\begin{proof}
Properties 1 and 3 are easily verified. For property 2, note that
\begin{align}
W'(x) = 2\frac{V'(x)x-V(x)}{x^{2}},
\end{align}
so we need to verify that $V'(x)<\frac{V(x)}{x}$. This is true by properties 1 and 4 in Proposition \ref{prop:vproperties}.
\end{proof}

In the case without disorder Conjecture \ref{con:phasedifference} can be proven.
\begin{proposition}
\label{prop:conjecture}
Fix $\psi_{1}=0$ and assume that $\mu_{1}=\mu_{2}=\delta_{0}$. Then the order parameters of the system are either $r_{1}, r_{2}=0$ or $r_{1}, r_{2}>0$ and $\psi\in \{0, \pi\}$.
\end{proposition}
\begin{proof}
Here the set of self-consistency equations \eqref{eq:thmselfcons} simplify to 
\begin{align}
r_{1} &= \int_{\mathbb{S}}\ddd \theta \,\cos(\psi_{1}-\theta)\,p_{1}(\theta),\label{eq:propselfcons1}\\
r_{2} &= \int_{\mathbb{S}}\ddd \theta \,\cos(\psi_{2}-\theta)\,p_{2}(\theta),\label{eq:propselfcons2}\\
0 &= \int_{\mathbb{S}}\ddd\theta \,\sin(\psi_{1}-\theta)\,p_{1}(\theta),\label{eq:propselfcons3}\\
0 &= \int_{\mathbb{S}}\ddd\theta \,\sin(\psi_{2}-\theta)\,p_{2}(\theta).\label{eq:propselfcons4}
\end{align}
Since the system is invariant under rotations we can set one of the average phase angles to zero. So take $\psi_{1}=0$ such that $\psi=\psi_{2}$. To determine which phase differences are possible we are left to solve 
\begin{align}
\label{eq:propanglesolve1}
0 &= \int_{\mathbb{S}}\ddd\theta \, \sin\theta\frac{\exp\Big[ Lr_{2}\cos(\psi_{2}-\theta) + Kr_{1}\cos\theta\Big]}{\int_{\mathbb{S}}\ddd\phi\,\exp\Big[Lr_{2}\cos(\psi_{2}-\theta) + Kr_{1}\cos\theta\Big]} \\
&= Lr_{2}\sin\psi \;W\Big(\sqrt{K^{2}r_{1}^{2} + L^{2}r_{2}^{2} + 2KLr_{1}r_{2}\cos\psi}\,\Big),\nonumber\\
\label{eq:propanglesolve2}
0 &= \int_{\mathbb{S}}\ddd\theta \, \sin(\psi_{2}-\theta)\frac{\exp\Big[ Lr_{1}\cos\theta + Kr_{2}\cos(\psi_{2} - \theta)\Big]}{\int_{\mathbb{S}}\ddd\phi\,\exp\Big[ Lr_{1}\cos\phi + Kr_{2}\cos(\psi_{2} - \phi)\Big]}\\
&= Lr_{1}\sin\psi \;W\Big(\sqrt{K^{2}r_{2}^{2} + L^{2}r_{1}^{2} + 2KLr_{1}r_{2}\cos\psi}\,\Big).\nonumber
\end{align}
Let us first consider the case when $r_{1}=0$. In this case \eqref{eq:propselfcons1} becomes 
\begin{equation}
\label{eq:propr1is0}
0 = \int_{\mathbb{S}}\ddd \theta \,\cos\theta\frac{\exp\Big[ Lr_{2}\cos(\psi_{2} - \theta)\Big]}{\int_{\mathbb{S}}\ddd\phi\,\exp\Big[ Lr_{2}\cos(\psi_{2} - \phi) \Big]}
\end{equation}
and \eqref{eq:propselfcons2} becomes
\begin{equation}
\label{eq:propr2}
r_{2} = \int_{\mathbb{S}}\ddd \theta \,\cos(\psi_{2}-\theta)\frac{\exp\Big[Kr_{2}\cos(\psi_{2} - \theta)\Big]}{\int_{\mathbb{S}}\ddd\phi\,\exp\Big[Kr_{2}\cos(\psi_{2} - \phi)\Big]} = V(Kr_{2}),
\end{equation}
which is exactly the self-consistency equation for the one-community noisy Kuramoto model without disorder, and can be divided into two cases: Either $K\leq 2$, in which case $r_{2}=0$, making $(r_{1}, r_{2})=(0, 0)$ the only stationary solution, or $K>2$, in which case there is a unique $r_{2}>0$ solving \eqref{eq:propr2}. By making the change of variable $\vartheta = \psi_{2}-\theta$ in \eqref{eq:propr1is0} and using the trigonometric identity $\cos(\psi_{2}-\vartheta) = \cos\psi_{2}\cos\vartheta + \sin\psi_{2}\sin\vartheta$ in \eqref{eq:propr1is0}, we see that \eqref{eq:propr1is0}, in this case, is only solved by $\psi_{2}=\frac{\pi}{2}$ or $\psi_{2}=\frac{3\pi}{2}$. In order to satisfy the self-consistency equations, these angles must satisfy \eqref{eq:propanglesolve1} and \eqref{eq:propanglesolve2} with $r_{1}=0$:
\begin{align}
0 &= \int_{\mathbb{S}}\ddd\theta \, \sin\theta\frac{\exp\Big[ Lr_{2}\cos(\psi_{2}-\theta)\Big]}{\int_{\mathbb{S}}\ddd\phi\,\exp\Big[Lr_{2}\cos(\psi_{2}-\theta) \Big]},\\
0 &= \int_{\mathbb{S}}\ddd\theta \, \sin(\psi_{2}-\theta)\frac{\exp\Big[Kr_{2}\cos(\psi_{2} - \theta)\Big]}{\int_{\mathbb{S}}\ddd\phi\,\exp\Big[Kr_{2}\cos(\psi_{2} - \phi)\Big]}.
\end{align}
The second equation is satisfied for all $\psi_{2}$, but the first equation is incompatible with $\psi_{2}=\frac{\pi}{2}$ as well as $\psi_{2}=\frac{3\pi}{2}$. so that the solution $r_{1}=0$ and $r_{2}>0$ is not possible, leaving only the solution $(r_{1}, r_{2})=(0, 0)$. Note that in this case the average angles are not well defined.

Let us next consider the case when $r_{1}>0$ (so that we must also have $r_{2}>0$). The allowed angles have to satisfy \eqref{eq:propanglesolve1} and \eqref{eq:propanglesolve2} simultaneously. These are satisfied only when $\sin\psi = 0$, so that $\psi\in\{0, \pi\}$.
\end{proof}

\begin{theorem}[\textbf{Critical line without disorder}]
\label{thm:zphasecritical}
Fix $\psi = \psi_{2}-\psi_{1}\in \{0, \pi\}$. Then the parameter space $\{(K, L): K, L\in \mathbb{R}^{2}\}$ splits into two regions:
\begin{enumerate}[a)]
\item In the region $K+L\cos\psi\leq 2$, there is precisely one solution: the unsynchronized solution $(r_{1}, r_{2}) = (0, 0)$.
\item In the region $K+L\cos\psi>2$, there are at least two solutions: the unsynchronized solution $(r_{1}, r_{2}) = (0, 0)$ and the symmetric synchronized solution $(r_{1}, r_{2}) = (r, r)$ for some $r\in (0, 1)$.
\end{enumerate}
\end{theorem}

\begin{proof}
For part a), note that $(0, 0)$ always solves the self-consistency equations in \eqref{eq:selfconsfinal}, due to property 1 of Proposition \ref{prop:wproperties} and the fact that $a_{1}, a_{2}, b_{1}$, $b_{2}$ are zero when $(r_{1}, r_{2})=(0, 0)$. The calculation given in the proof of Proposition \ref{prop:conjecture} when $r_{1}=0$ shows that a solution of the form $r_{1}=0$ and $r_{2}>0$ is not possible, and due to symmetry the same is true for solutions with $r_{2}=0$ and $r_{1}>0$. To have strictly positive $r_{1}, r_{2}$, we use property 5 in Proposition \ref{prop:vproperties} to get 
\begin{align}
r_{1} < \frac{Kr_{1}+Lr_{2}\cos\psi}{2},\nonumber\\
r_{2} < \frac{Kr_{2}+Lr_{1}\cos\psi}{2}.
\end{align}
Adding these equations, we get
\begin{align}
\label{eq:critlineproof}
K+L\cos\psi>2,
\end{align}
which is the condition to have positive synchronized solutions and defines the critical line. Let us next consider the case $\psi=0$ and $r_{1}, r_{2}>0$. Then the self-consistency equations in \eqref{eq:selfconsfinal} reduce to 
\begin{align}
r_{1} &= \frac{(Kr_{1} + Lr_{2})}{2} W(Kr_{1}+Lr_{2}) = V(Kr_{1}+Lr_{2}),\label{eq:selfconsreduced1}\nonumber\\
r_{2} &= \frac{(Kr_{2} + Lr_{1})}{2} W(Kr_{2}+Lr_{1}) = V(Kr_{2}+Lr_{1}).
\end{align}
If we consider symmetric solutions so that $r_{1}=r_{2}=r$, then these two equations are identical and correspond to the self-consistency equation for the one-community noisy Kuramoto model with the replacement $2K\rightarrow K+L$, which has a positive solution when $K+L>2$. The same can be done when $\psi=\pi$ and yields $K-L>2$ as critical condition. 
\end{proof}

It is tempting to conclude that the two-community model is the same as the one-community model with the replacement $2K\to K+L\cos\psi$. This is, however, not the case as we will see in Section \ref{sec:nonsym}. 

\begin{figure}
\centering
\includegraphics[scale=1.7]{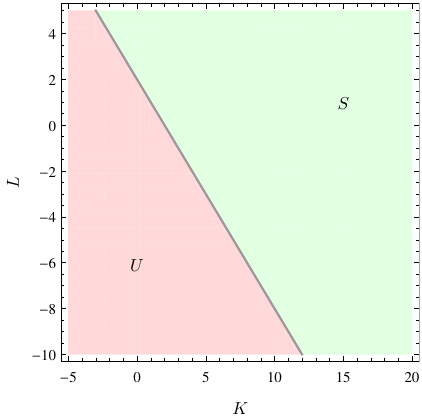}
\includegraphics[scale=1.7]{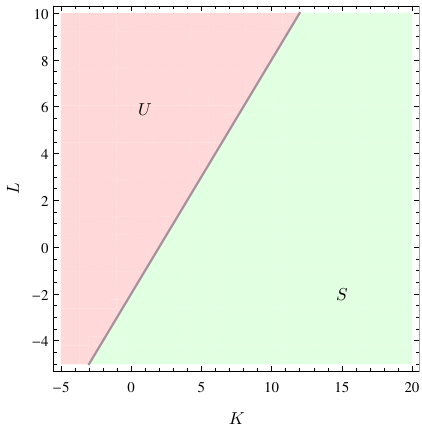}
\caption{Regions appearing in Theorem \ref{thm:zphasecritical} $\psi=0$ (left) $\psi=\pi$ (right). Part a): the red region (labeled by a $U$); part b): the green region (labeled by an $S$). }
\label{fig:phasediagrampsi}
\end{figure}

\subsection{With disorder}
\label{sec:withdisorder}
In this section we identify the critical line when we include disorder. We simplify the system by taking the distributions from which the natural frequencies are drawn in the two communities to be the same, i.e., $\mu_{1}=\mu_{2}=\mu$. Then the self-consistency equations in \eqref{eq:thmselfcons} read
\begin{align}
\label{eq:selfconsdisorder}
r_{1} &= V_{1}^{\mu}(r_{1}, r_{2})=\int_{\mathbb{S}}\ddd\theta\int_{\mathbb{R}}\mu(\ddd\omega)\cos(\psi_{1}-\theta)\, p_{1}(\theta, \omega),\nonumber\\
r_{2} &= V_{2}^{\mu}(r_{1}, r_{2})=\int_{\mathbb{S}}\ddd\theta\int_{\mathbb{R}}\mu(\ddd\omega)\cos(\psi_{2}-\theta)\, p_{2}(\theta, \omega),\\
0 &= U_{1}^{\mu}(r_{1}, r_{2}) := \int_{\mathbb{S}}\ddd\theta\int_{\mathbb{R}}\mu(\ddd\omega)\sin(\psi_{1}-\theta)p_{1}(\theta, \omega),\nonumber\\
0 &= U_{2}^{\mu}(r_{1}, r_{2}) := \int_{\mathbb{S}}\ddd\theta\int_{\mathbb{R}}\mu(\ddd\omega)\sin(\psi_{2}-\theta)p_{2}(\theta, \omega).\nonumber
\end{align}
In light of Conjecture \ref{con:phasedifference} we will restrict the following theorem to the two cases $\psi=0$ and $\psi=\pi$. Define
\begin{align}
\chi=\int_{\mathbb{R}}\mu(\ddd\omega)\,\frac{1}{2(1+4\omega^{2})}.
\end{align}

\begin{conjecture}[\textbf{Critical line with disorder}]
\label{thm:critcondisorder}
Fix $\psi=\psi_{2}-\psi_{1}\in \{0, \pi\}$. If the disorder in the two communities is drawn from a symmetric unimodal distribution $\mu$, then the parameter space $\{(K, L):K, L \in \mathbb{R}^{2}\}$ splits into two regions:
\begin{enumerate}[a)]
\item In the region $K+L\cos\psi\leq \chi^{-1}$, there is precisely one solution: the unsynchronized solution $(r_{1}, r_{2}) = (0, 0)$.
\item In the region $K+L\cos\psi>\chi^{-1}$, there are at least two solutions: the unsynchronized solution $(r_{1}, r_{2}) = (0, 0)$ and the symmetric synchronized solution $(r_{1}, r_{2}) = (r, r)$ for some $r\in (0, 1)$. 
\end{enumerate}
\end{conjecture}
\begin{hproof}
Following the method used in \cite{SA88} for the one-community model, we Taylor expand the self-consistency equations for $r_{1}$ and $r_{2}$ in the two variables $r_{1}$ and $r_{2}$. The equations in \eqref{eq:selfconsdisorder} read, to first order,
\begin{align}
r_{1} &= V_{1}^{\mu}(0, 0) +\partial_{r_{1}}V_{1}^{\mu}(r_{1}, r_{2})|_{(r_{1}, r_{2}) = (0, 0)}r_{1} + \partial_{r_{2}}V_{1}^{\mu}(r_{1}, r_{2})|_{(r_{1}, r_{2}) = (0, 0)}r_{2} + O(r_{1}^{2}+r_{2}^{2}),\\
r_{2} &= V_{2}^{\mu}(0, 0) +\partial_{r_{1}}V_{2}^{\mu}(r_{1}, r_{2})|_{(r_{1}, r_{2}) = (0, 0)}r_{1} + \partial_{r_{2}}V_{2}^{\mu}(r_{1}, r_{2})|_{(r_{1}, r_{2}) = (0, 0)}r_{2} + O(r_{1}^{2}+r_{2}^{2}).\nonumber
\end{align}
We can verify that $V_{1}^{\mu}(0, 0)= V_{2}^{\mu}(0, 0)=0$, and calculate the derivatives at zero. This leads to 
\begin{align}
r_{1} &= r_{1}K\chi + r_{2}\int_{\mathbb{R}}\mu(\ddd\omega)\,\frac{L(\cos(\psi_{1}-\psi_{2})+2\omega\sin(\psi_{1}-\psi_{2}))}{2(1+4\omega^{2})}+ O(r_{1}^{2}+r_{2}^{2}),\\
r_{2} &= r_{2}K\chi + r_{1}\int_{\mathbb{R}}\mu(\ddd\omega)\,\frac{L(\cos(\psi_{2}-\psi_{1})+2\omega\sin(\psi_{2}-\psi_{1}))}{2(1+4\omega^{2})}+ O(r_{1}^{2}+r_{2}^{2})\nonumber.
\end{align}
Adding these equations, we get
\begin{align}
\label{eq:disordercritproof}
r_{1}+r_{2} =& \;(r_{1}+r_{2})(K+L\cos(\psi_{1}-\psi_{2}))\,\chi\nonumber\\
 &+ (r_{2}-r_{1})2L\sin(\psi_{1}-\psi_{2})\int_{\mathbb{R}}\mu(\ddd\omega)\,\frac{\omega}{2(1+4\omega^{2})}+ O(r_{1}^{2}+r_{2}^{2}).
\end{align}
Since we are considering the case where $\mu$ is symmetric, the last term vanishes and we obtain the critical line in Theorem~\ref{thm:critcondisorder}. This shows that below the critical line the self-consistency equations are a contraction, making $(r_{1}, r_{2})=(0, 0)$ a fixed point. In order to show that solutions of the form $r_{1}=0$ and $r_{2}>0$ are not possible, we would have to repeat the calculation used in the proof of Proposition \ref{prop:conjecture} for the general case. This turns out to be non-trivial, but we expect that it is possible to prove this for symmetric, unimodal $\mu$ by proving that $p_{1}(\theta+\psi_{2}, \omega)=p_{1}(-\theta+\psi_{2}, -\omega)$ (in the case that $\psi_{1}=0$) and using this symmetry to show that the first and third equation in \eqref{eq:thmselfcons} cannot be simultaneously satisfied when $r_{1}=0$ and $r_{2}>0$. If $\mu$ is symmetric and unimodal, then it is conjectured that the analog of $V_{1}^{\mu}(r_{1}, r_{2})$ and $V_{2}^{\mu}(r_{1}, r_{2})$ in the one-community noisy Kuramoto model is concave [see Conjecture [3.12], Chapter 3 in \cite{L12a}]. We assume that this conjecture also holds in this case for both $V_{1}^{\mu}(r_{1}, r_{2})$ and $V_{2}^{\mu}(r_{1}, r_{2})$, at least for symmetric solutions. In the case $\psi=0$ the symmetric solution $r_{1}=r_{2}>0$ reduces the system of self-consistency equations in \eqref{eq:selfconsdisorder} to a single equation that is analogous to the one-community noisy Kuramoto model self-consistency equation \cite[Proposition 3.10, Chapter 3]{L12a} with the replacement $K\rightarrow K+L$. In the case $\psi=\pi$ we can perform a change of variable in the integral of the second line \eqref{eq:selfconsdisorder}, namely, $\phi = \psi_{2}-\theta$, to see that the equations again reduce to the equation for the one-community case with the replacement $K\rightarrow K-L$. Thus, we see that in both cases we can apply the conjecture in \cite[Conjecture 3.12]{L12a} to ensure that the line $K+L\cos\psi = \chi^{-1}$ is the critical condition for symmetrically synchronized solutions, which settles the conditions in a) and b).

For part b), we must still show that the symmetric solution is possible above the critical line. Due to the reduction of the system to the one-community noisy Kuramoto model, both for $\psi=0$ and for $\psi=\pi$, we see that the symmetric solution indeed exists above the critical line.
\end{hproof}

\section{Bifurcation of non-symmetric solutions}
\label{sec:nonsym}
In this section we consider the system with the same parameter specifications and simplifications as in Section \ref{sec:critcon}, but without disorder and with $\psi=0$. The analysis with $\psi=\pi$ carries over after the replacement $L\to -L$ in the self-consistency equations in \eqref{eq:selfconsreduced1} (the resulting modified phase diagram is shown in the right panel of Fig.~\ref{fig:asymphasediagram}). The proofs in this section rely on numerics. 

The self-consistency equations can be visualized as a vector field, in which the solutions to the equations appear as fixed points, by plotting
\begin{equation}
\label{eq:vectorfield}
\vec{V}_{r_{1}, r_{2}} = (V(Kr_{1}+Lr_{2})-r_{1}, V(Kr_{2} + Lr_{1})-r_{2}).
\end{equation}
For a certain range of parameters non-symmetric solutions appear, as seen in Fig.~\ref{fig:unstablevectorfield}. The non-symmetric solutions appear to be saddle-points, having a stable and an unstable manifold under the vector field representing the self-consistency equations. Note that this vector field does not represent the dynamics of the system, since the self-consistency equations contain only the stationary densities. By plotting the possible solutions as functions of $K$ while keeping $L$ fixed, we see that the non-symmetric solutions bifurcate from the symmetric solutions, as is seen in Fig.~\ref{fig:rsolutiontypes} for the case where $L=-2$. The symmetric solutions correspond to equal amounts of synchronization in the two communities. This is also the only solution possible between $K=4$ and $K=4.9953\ldots$. At $K=4.9953\ldots$, the non-symmetric solutions appear, corresponding to one community having a larger synchronization level than the other community. Due to the symmetry of the system, both communities can have a higher level of synchronization in the non-symmetric solution.

\begin{figure}
\centering
\includegraphics[scale=2.0]{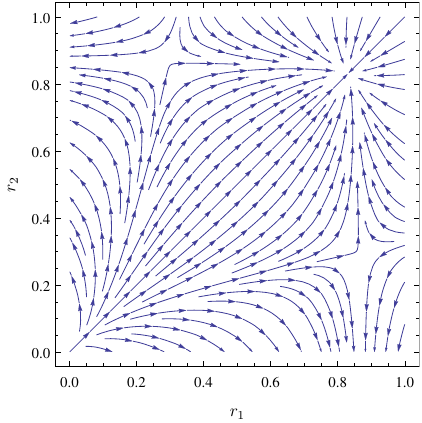}
\caption{Self-consistency vector field \eqref{eq:vectorfield} for $K=5$ and $L=-1$.}
\label{fig:unstablevectorfield}
\end{figure}

\begin{figure}
\centering
\includegraphics[scale=2.0]{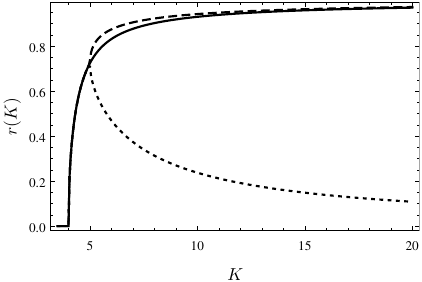}
\caption{Solutions to the self-consistency equations in \eqref{eq:selfconsfinal} for different values of $K$ when $L=-2$. Drawn are the pairs of symmetric solutions (solid), and the pairs of non-symmetric solution (dashed and dotted).}
\label{fig:rsolutiontypes}
\end{figure}

In Section~\ref{sec:existandchar} we prove a necessary and sufficient condition for the existence of non-symmetric solutions. In Section~\ref{sec:ordering} we show that the non-symmetric solutions are ordered and are such that the symmetric solution is wedged in between the two non-symmetric solutions. In Section~\ref{sec:properties} we analyze the (asymptotic) properties of the bifurcation line as well as the synchronization level along the bifurcation line.

\subsection{Existence and characterization of non-symmetric solutions}
\label{sec:existandchar}
 
\begin{theorem}[\textbf{Characterization of the  bifurcation line}]
The existence of non-symmetric solutions requires $L<0$, in which case the bifurcation point $K^{*}=K^{*}(L)$ is the unique solution to the equation
\begin{align}
\sqrt{1 - \frac{2K}{K^{2}-L^{2}}} = V\Bigg((K+L)\sqrt{1 - \frac{2K}{K^{2}-L^{2}}}\;\Bigg)\label{eq:criticalselfcon},
\end{align} 
and the synchronization level at the bifurcation point is given by
\begin{align}
r^{*}(K^{*}, L) = \sqrt{1 - \frac{2K^{*}}{{K^{*}}^{2}-L^{2}}}.\label{eq:criticalr}
\end{align}
\end{theorem}
\begin{proof}
We assume that a non-zero symmetric solution exists, so that $r_{1} = r_{2}= r$ and $r = V((K+L)r)$, which is the case when $K+L>2$. Let $(K^{*},r^{*})$ be a bifurcation point for fixed $L$. We will show via a perturbation argument that this bifurcation point exists and is unique. At the bifurcation point the non-symmetric solutions split off from the symmetric solution since $V$ is  continuous. This allows us to perform a perturbation around $r^{*}$, namely,
\begin{align}
r^{*}+\epsilon &= V(K(r^{*}+\epsilon) + L(r^{*}-\delta)),\\
r^{*}-\delta &= V(K(r^{*}-\delta) + L(r^{*}+\epsilon)),
\end{align}
where $\epsilon$ and $\delta$ are small, either positive or negative, and are related, as will be shown shortly. We Taylor expand around the point $(K+L)r^{*}$ and use $r^{*} = V((K+L)r^{*})$, to get
\begin{align}
\label{eq:perturbedselfcon}
\epsilon \sim (K\epsilon -L\delta)V'((K+L)r^{*}),\quad -\delta \sim (L\epsilon - K\delta)V'((K+L)r^{*}), \quad \epsilon, \delta\downarrow 0,
\end{align}
where by $\sim$ (here and in the rest of the paper) we mean that the ratio tends to 1 asymptotically. Abbreviate $C^{*}=V'((K+L)r^{*})$. Then the equations in \eqref{eq:perturbedselfcon} combine to give
\begin{align}
\epsilon \sim \Big(\frac{LC^{*}}{KC^{*}-1}\Big)^{2}\epsilon,
\end{align}
which implies
\begin{align}
LC^{*} = \pm(KC^{*} - 1),
\end{align}
and $\epsilon\sim \pm \delta$. Using the negative sign would require the following two equations to be satisfied:
\begin{align}
r^{*} &= V((K+L)r^{*}),\\
\frac{1}{K+L} &= V'((K+L)r^{*}).
\end{align} 
However, these equations cannot be satisfied simultaneously with $r^{*}>0$. Indeed, the first finds the intersection point of $V$ with the line of slope $\frac{1}{K+L}$ passing through zero. But due to properties 1 and 4 we know that $V$ has slope $\frac{1}{K+L}$ before this intersection point. Thus, the two equations that must be satisfied at the bifurcation point are
\begin{align}
r^{*} &= V((K+L)r^{*})\label{eq:proofcritselfcon},\\
\frac{1}{K-L} &= V'((K+L)r^{*}).\label{eq:proofcritselfVprime}
\end{align} 
For fixed $L$, these equations determine both the value $r^{*}=r^{*}(L)$ of the synchronization level at the bifurcation point and the internal coupling strength $K^{*}=K^{*}(L)$ at which the bifurcation occurs. The first equation finds the intersection point of $V$ and the line with slope $\frac{1}{K+L}$ passing through zero. The second equation requires the derivative of $V$ at this point to be $\frac{1}{K-L}$. Due to the concavity of $V$ (Property 4 of Proposition \ref{prop:vproperties}), this gives the relation 
\begin{align}
\frac{1}{K+L}>\frac{1}{K-L},
\end{align}
which implies that $L<0$, as claimed. To visualize the procedure for determining the bifurcation point, we plot the appropriate lines in Fig.~\ref{fig:bifurccomic}. It is clear that the slope of the thickly dashed line must be less than that of the solid line, which gives $L<0$.

\begin{figure}
\centering
\includegraphics[scale=2.0]{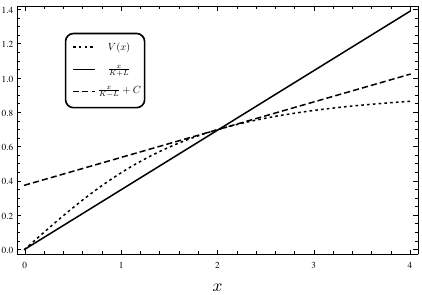}
\caption{Visualization of the procedure to determine the bifurcation point. (Here $C$ is a constant determined in order to plot the tangent line and is not $C^{*}$.)}
\label{fig:bifurccomic}
\end{figure}

We can find an expression for the derivative of $V$ in \eqref{eq:Vdefinition} by writing
\begin{align}
V'(x) = \frac{\int_{\mathbb{S}}\ddd \theta\,\cos^{2}\theta\, \eee^{x \cos\theta}}{\int_{\mathbb{S}}\ddd \theta\, \eee^{x \cos\theta}} - V^{2}(x).
\end{align}
For the first term in the right-hand side we can use the identity from  \cite[Eq.~(2.21)]{G10}, so that in our case
\begin{align}
\label{eq:giacominvprime}
V'((K+L)r^{*}) = 1-\frac{1}{K+L} - (r^{*})^{2},
\end{align}
where we have used \eqref{eq:proofcritselfcon} for the second term. This reduces \eqref{eq:proofcritselfVprime} to 
\begin{align}
\label{eq:proofcriticalr}
r^{*}(K, L) = \sqrt{1 - \frac{2K}{K^{2}-L^{2}}}.
\end{align}
To find $r^{*}=r^{*}(K^{*}, L)$, we must find $K^{*} = K^{*}(L)$ that solves \eqref{eq:proofcritselfcon}. Substituting  \eqref{eq:proofcriticalr} into \eqref{eq:proofcritselfcon}, we obtain \eqref{eq:criticalselfcon}.

We will first prove that, given $r$, there is a unique $K^{*}$. In order to do this, we solve the equation
\begin{align}
r = \sqrt{1-\frac{2K}{K^{2}-L^{2}}}
\end{align}
for $L$ to find
\begin{align}
\label{eq:Lsolved}
L = -\sqrt{K^{2} - \frac{2K}{1-r^{2}}},
\end{align}
where we have taken the negative since we are dealing with the case $L<0$. In order to have a real solution, we require 
\begin{align}
\label{eq:Kcondition}
K> \frac{2}{1-r^{2}}.
\end{align}
The equation for the bifurcation point in \eqref{eq:criticalselfcon} reads 
\begin{align}
\label{eq:frselfcons}
V(f_{r}(K)r) = r,
\end{align}
where 
\begin{align}
f_{r}(K) = K - \sqrt{K^{2} - \frac{2K}{1-r^{2}}}.
\end{align}
Clearly, $K\mapsto f_{r}(K)$ is strictly decreasing on $(\frac{2}{1-r^{2}}, \infty)$ for $r\in (0, 1)$. Since $x\mapsto V(x)$ is strictly increasing, $K\mapsto V(f_{r}(K)r)$ is strictly decreasing on $(\frac{2}{1-r^{2}}, \infty)$. However, in order to satisfy \eqref{eq:frselfcons} with $r\in (0, 1)$, by property 5 in Proposition \ref{prop:vproperties}, we must have 
\begin{align}
f_{r}(K)>2,
\end{align}
i.e.,
\begin{align}
&K\in\Big(\frac{2}{1-r^{2}},\frac{2(1-r^{2})}{1-2r^{2}}\Big),\quad r\in\Big(0, \frac{1}{\sqrt{2}}\Big),\\
&K\in\Big(\frac{2}{1-r^{2}},\infty\Big),\quad r\in\Big[\frac{1}{\sqrt{2}}, 1\Big).
\end{align}
Moreover,
\begin{align}
\lim_{K\rightarrow \infty}f_{r}(K) = \frac{1}{1-r^{2}}.
\end{align}  
For fixed $r\in(0, 1/\sqrt{2})$, $V(f_{r}(K)r)$ decreases from $V(\frac{2r}{1-r^{2}})$ to $V(2r)$ as $K$ increases from $\frac{2}{1-r^{2}}$ to $\frac{2(1-r^{2})}{1-2r^{2}}$ while for $r\in[1/\sqrt{2}, 1)$, $V(f_{r}(K)r)$ decreases from $V(\frac{2r}{1-r^{2}})$ to $V(\frac{r}{1-r^{2}})$ as $K$ increases from $\frac{2}{1-r^{2}}$ to $\infty$. In order to prove uniqueness, we need to show that
\begin{align}
&V\Big(\frac{2r}{1-r^{2}}\Big)>r>V(2r),\quad r\in (0, 1/\sqrt{2}),\label{eq:bounds1}\\
&V\Big(\frac{2r}{1-r^{2}}\Big)>r>V\Big(\frac{r}{1-r^{2}}\Big),\quad  r\in [1/\sqrt{2}, 1)\label{eq:bounds2}.
\end{align}
Uniqueness follows because is be a unique $K^{*}$ satisfying \eqref{eq:proofcriticalr}, due to $V$ decreasing continuously from the upper to the lower bounds in \eqref{eq:bounds1} and \eqref{eq:bounds2} and the line $r$ being wedged between the bounds (note that $r\mapsto V(f_{r}(K)r)$ intersects $r$ exactly once). The curves $V(\frac{2r}{1-r^{2}})$, $r$, $V(\frac{r}{1-r^{2}})$ are plotted numerically in Fig.~\ref{fig:uniquenessr}, which shows that the bounds in \eqref{eq:bounds2} and the upper bound in \eqref{eq:bounds1} hold for all $r\in (0, 1)$. The lower bound in \eqref{eq:bounds1} is immediate from property 5 in Proposition \ref{prop:vproperties}. Indeed, we see that the bifurcation point exists and that $K^{*}$ is unique given $r$. We will show later that $r^{*}$ is also unique given $K$ by showing that $\frac{\partial r^{*}}{\partial K}>0$ in Theorem ~\ref{lemma:rprop}.
\end{proof}

\begin{figure}[H]
\centering
\includegraphics[scale=2.0]{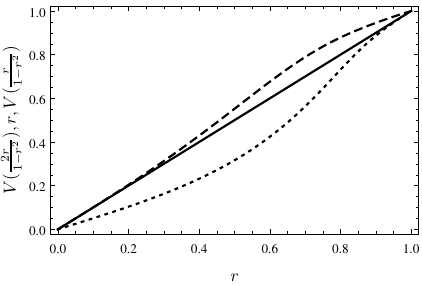}
\caption{Plot via MATHEMATICA of $V(\frac{2r}{1-r^{2}})$ (dashed), $r$ (solid) and $V(\frac{r}{1-r^{2}})$ (dotted) as functions of $r$.}
\label{fig:uniquenessr}
\end{figure} 

The uniqueness of the bifurcation point corroborates the picture in Fig.~\ref{fig:rsolutiontypes}.

\begin{remark}
Note that \eqref{eq:criticalselfcon} can also be solved for $L^{*}=L^{*}(K)$. The way this should be understood is that, after one of the variables $K$ and $L$ is fixed, the bifurcation point for the other variable is determined. A plot of the bifurcation point as a function of $K$ and $L$ is shown in Fig.~\ref{fig:bifurcationpoint}. 
\end{remark}

\begin{figure}[H]
\centering
\includegraphics[scale=2.0]{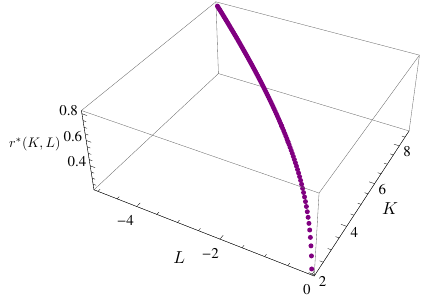}
\caption{Plot of $(K, L)\mapsto r^{*}(K, L)$ along the critical line.}
\label{fig:bifurcationpoint}
\end{figure}

\subsection{Ordering of non-symmetric solutions}
\label{sec:ordering}

Due to the symmetry of the system, if $(r_{1}, r_{2})$ is a solution to \eqref{eq:selfconsfinal} with $\psi=0$, then so is $(r_{2}, r_{1})$. When non-symmetric solutions exist, we have the following ordering of the synchronization levels in the two communities.
\begin{theorem}[\textbf{Ordered solutions}]
Fix $L$ and take $K>K^{*}$ where $K^{*}$ is the bifurcation point obtained by solving \eqref{eq:criticalselfcon}. Furthermore take only positive solutions so that $r_{1}, r_{2}, r>0$. Without loss of generality, consider a non-symmetric solution with $r_{1}>r_{2}$. Then
\begin{align}
r_{2}<r<r_{1}.
\end{align}
\end{theorem}
\begin{proof}
The symmetric solution $r$ solves the equation
\begin{align}
r = V(r(K+L)).
\end{align}
To prove that $r<r_{1}$, we consider the self-consistency equation \eqref{eq:selfconsreduced1} for $r_{1}$,
\begin{align}
r_{1} = V\Big(r_{1}\Big(K+L\frac{r_{2}}{r_{1}}\Big)\Big),
\end{align}
and recall that we must have $L<0$ for non-symmetric solutions to exist. Since $\frac{r_{2}}{r_{1}}<1$, we know that $K+L\frac{r_{2}}{r_{1}}>K+L$ and, due to the fact that $x\mapsto V(x)$ is strictly increasing, also $r<r_{1}$. Note that we are not quantifying the difference $r_{1}-r_{2}$. The strict inequality follows purely from the fact that $\frac{r_{2}}{r_{1}}<1$, making it impossible to match the solutions for $r$ and $r_{1}$. Similarly, we can show that $r_{2}<r$.
\end{proof}

\subsection{Properties of the bifurcation line}
\label{sec:properties}
We cannot solve \eqref{eq:proofcritselfcon} analytically for $K^*$. We can, however, plot \eqref{eq:proofcritselfcon} numerically, which refines the phase diagram in Fig.~\ref{fig:phasediagrampsi} for $\psi = 0$, as shown in Fig.~\ref{fig:asymphasediagram}. In this section we first list some basic properties of $r^{*}(K)$ and its derivatives, defined as the solution of \eqref{eq:proofcritselfcon} when we eliminate $L$ with the help of \eqref{eq:Lsolved}. After that we state a theorem on the asymptotic properties of the bifurcation line $L^{*}(K)$ defined implicitly by \eqref{eq:criticalselfcon}.  

\begin{figure}
\centering
\includegraphics[scale=1.7]{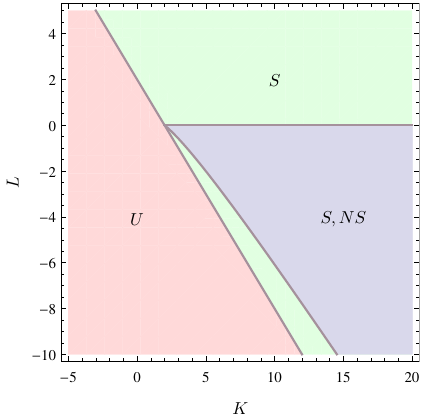}
\includegraphics[scale=1.7]{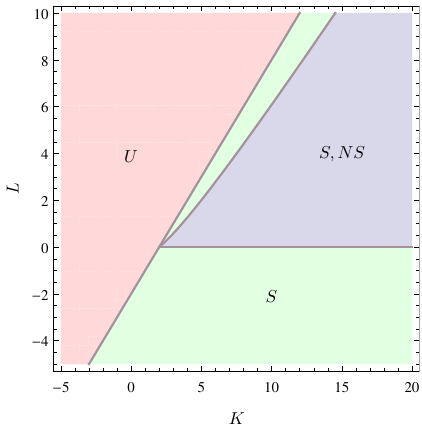}
\caption{In the light red region there is one pair of solutions: unsynchronized. In the light green region there are two pairs of solutions: unsynchronized and symmetric synchronized. In the light blue region there are three pairs of solutions: unsynchronized, symmetric synchronized and non-symmetric synchronized.}\label{fig:asymphasediagram}
\end{figure} 

\begin{theorem}[\textbf{Properties of $K\mapsto r^{*}(K)$}]\ \\
\label{lemma:rprop}
\vspace{-0.5cm}
\begin{enumerate}
\item $\lim_{K\downarrow 2}r^{*}(K) = 0$.
\item $\lim_{K\rightarrow\infty}r^{*}(K) = 1$.
\item $r^{*}(K)\sim \sqrt{\frac{K-2}{2}}\quad \text{as }K\downarrow 2$.
\item $1-r^{*}(K)\sim \frac{1}{2\sqrt{K}}\quad \text{as }K\rightarrow \infty$.
\item $\frac{\partial r^{*}(K)}{\partial K} > 0\quad \text{for all } K > 2$.
\item $\frac{\partial^{2} r^{*}(K)}{\partial K^{2}} < 0\quad \text{for all } K > 2$.
\end{enumerate}
\end{theorem}
\begin{proof}
We use \eqref{eq:Kcondition} to get 
\begin{align}
0\leq r^{*}(K) < \sqrt{\frac{K-2}{K}},
\end{align}
from which property 1 follows by taking the limit $K\downarrow 2$. The inequality in \eqref{eq:Kcondition} also implies
\begin{align}
\lim_{r\uparrow 1}K^{*}(r) = \infty.
\end{align}
If both $K\mapsto r^{*}(K)$ and $r\mapsto K^{*}(r)$ are continuous, then property 2 follows. To show that $K\mapsto r^{*}(K)$ is continuous, we apply the implicit function theorem (IFT) to calculate the derivative  
\begin{equation}
h(K, r^{*}) = V\Bigg(\Big(K-\sqrt{K^{2}-\frac{2K}{1-(r^{*})^{2}}}\Big)r^{*}\Bigg) -r^{*},
\end{equation}
which we find by rewriting \eqref{eq:frselfcons}. From the conditions for the IFT \cite{KP02}, we have that, in order for $K\mapsto r^{*}(K)$ to be continuous, we need 
that
\begin{equation}
2K [(r^{*})^{2}-1 + K(1-(r^{*})^{2})^{2}]\neq 0,
\end{equation}
which we obtain by differentiating $h(K, r)$ with respect to $r$ and setting the derivative to zero. From this we obtain the following bound on $r^{*}(K)$:
\begin{equation}
\label{eq:IFTrbound}
r^{*}(K)>\sqrt{1-\frac{1-\sqrt{1+4K}}{2K}} = r_{-}^{*}.
\end{equation}
In order to rigorously show that this bound is satisfied, we can use the sequence of (iteratively defined) upper bounds $x u_{1}^{(k)}(x), k\in \mathbb{N}_{0}$, for $V(x)$ given in \cite[Theorem 4]{S11}, which converge to $V(x)$ as $k\rightarrow\infty$. Here we will use 
\begin{equation}
l_{\nu}^{(1)} = \Bigg(\nu-\frac12 + \sqrt{\Big(\nu+\frac12\Big)^{2}+x^{2}}\Bigg)^{-1},
\end{equation}
as suggested in\cite[Equation (22)]{S11}. If substitution of the right-hand side of \eqref{eq:IFTrbound} for $r$ into $f_{r}(K)r u_{1}^{(k)}(f_{r}(K)r)-r$ makes it less than 0, then we know that the bound in \eqref{eq:IFTrbound} is satisfied. To see why, note that then
\begin{equation}
f_{r_{-}^{*}}(K)r_{-}^{*} u_{1}^{(k)}(f_{r_{-}^{*}}(K)r_{-}^{*})-r_{-}^{*} > h(K, r_{-}^{*}).
\end{equation} 
Now, if $f_{r_{-}^{*}}(K)r_{-}^{*}u_{1}^{(k)}(f_{r_{-}^{*}}(K)r_{-}^{*})-r_{-}^{*}<0$ for all $K$, then so is $h(K, r_{-}^{*})$, so that $r_{-}^{*}$ does not satisfy $h(K, r^{*})=0$ and the solution satisfies $r^{*}> r_{-}^{*}$. Using $xu_{1}^{(k)}(x)$ with $k=2$, 
\begin{equation}
xu_{1}^{(2)}(x) = \frac{x}{2+\frac{x^{2}}{\frac{3}{2}+\sqrt{(\frac{5}{2})^{2}+x^{2}}}},
\end{equation}
as an upper bound, we get that the bound in \eqref{eq:IFTrbound} is at least satisfied for $K\in (2, K_{k=2})$, where $K_{k=2}=15.8684$. By increasing $k$, we see that the upper bound of this interval increases and we expect that in the limit as $k\rightarrow \infty$, \eqref{eq:IFTrbound} is satisfied on $K\in (2,\infty)$. Numerically, we indeed see that this bound is satisfied, as shown in Fig.~\ref{fig:kasr} (this figure shows that $K^{*}(r) < \frac{2-r^{2}}{(1-r^{2})}$, which is the same as \eqref{eq:IFTrbound}). For the continuity of $r\mapsto K^{*}(r)$ we require, again by the condition for the IFT, that
\begin{equation}
\partial_{K^{*}}h(K^{*}, r) \neq 0,
\end{equation}
which is satisfied by all $K^{*}>0$ and $r^{*}\in (0, 1)$, so that property 2 is proved.

We know that $\lim_{K\downarrow 2} r(K)=0$ (by property 1), so that we can expand $V$ around 0 in the self-consistency equation \eqref{eq:frselfcons}. This leads to 
\begin{align}
\lim_{K\downarrow 2}f_{r^{*}(K)}(K) = 2.
\end{align}
The corresponding asymptotic equation can be solved for $r^{*}(K)$ to obtain property 3. Property 4 follows from a similar calculation, by using the expansion of $V$ around infinity, and gives 
\begin{align}
1-r^{*}(K) \sim \frac{1}{2f_{r^{*}(K)}(K)r^{*}(K)}.
\end{align}
This equation gives rise to a cubic polynomial in $r^{*}(K)$, which can be solved and gives 
\begin{align}
\label{eq:asymcomplex}
1-r^{*}(K)\sim \frac{1}{3} - \frac{(1-\mathrm{i}\sqrt{3})K}{3B} - \frac{(1+\mathrm{i}\sqrt{3}) B}{12 K},
\end{align}
where
\begin{align}
B = \Big(8K^{3} - 27K^{2} +3\sqrt{3}\sqrt{27K^{4}-16K^{5}}\Big)^{1/3}.
\end{align}
The complex parts in the right-hand side of \eqref{eq:asymcomplex} compensate one another, making it real. Taking only the leading order terms in $K$, we obtain the asymptotics in property 4. We can calculate $\partial_{K}r^{*}(K)$ by differentiating  \eqref{eq:frselfcons}, i.e.,
\begin{align}
\label{eq:rderivative}
\frac{\partial r^{*}(K)}{\partial K} = \frac{cr^{*}\Big(\sqrt{K^{2}-\frac{2K}{1-{r^{*}}^{2}}}- K -\frac{1}{(1-{r^{*}}^{2})}\Big)}{\sqrt{K^{2}-\frac{2K}{1-{r^{*}}^{2}}} - \frac{2cK{r^{*}}^{2}}{(1-{r^{*}}^{2})^{2}} + c\sqrt{K^{2}-\frac{2K}{1-{r^{*}}^{2}}}\Big(\sqrt{K^{2}-\frac{2K}{1-{r^{*}}^{2}}}-K\Big)},
\end{align}
where in the right-hand side we have abbreviated $r^{*}=r^{*}(K)$, and
\begin{align}
c = V'(f_{r^{*}}(K)r^{*}). 
\end{align}
It follows from \eqref{eq:proofcritselfVprime} that
\begin{align}
c = \frac{1}{K-L}=\frac{1}{K+\sqrt{K^{2}-\frac{2K}{1-{r^{*}}^{2}}}},
\end{align}
which simplifies \eqref{eq:rderivative} to
\begin{align}
\label{eq:rderivativesimple}
\frac{\partial r^{*}(K)}{\partial K} = \frac{r^{*}(1-{r^{*}}^{2})\Big\{\Big(K-\sqrt{K^{2}-\frac{2K}{1-{r^{*}}^{2}}}\Big)(1-{r^{*}}^{2})-1\Big\}}{2K\{2-{r^{*}}^{2}-K(1-{r^{*}}^{2})^{2}\}}.
\end{align}
Due to the inequality in \cite[Equation~(2.4)]{G10}, we have that
\begin{align}
\frac{1}{f_{r^{*}}(r^{*}(K))} < 1-r^{*}(K)^{2} < \frac{2}{f_{r^{*}}(r^{*}(K))},
\end{align}  
which makes the numerator positive. The denominator becomes zero when 
\begin{align}
K = \frac{2-{r^{*}}^{2}}{(1-{r^{*}}^{2})^{2}}. 
\end{align}
Rewriting the lower bound for $r^{*}$ in \eqref{eq:IFTrbound}, we get 
\begin{equation}
\label{eq:Kupperbound}
K < \frac{2-{r^{*}}^{2}}{(1-{r^{*}}^{2})^{2}},
\end{equation} 
which ensures that the denominator of \eqref{eq:rderivative} is never zero. For values of $K$ satisfying \eqref{eq:Kupperbound} the derivative is positive. This we find by substituting a pair of values $r^{*}(K), K$, calculated numerically, into \eqref{eq:rderivative}, and proves property 5 because the derivative does not change sign in the range of $K$. To prove property 6, we take the derivative with respect to $K$ of \eqref{eq:rderivativesimple} and substitute the expression for the first derivative. This leads to a lengthy equation with denominator
\begin{align}
4K^{2}\sqrt{K^{2}-\frac{2K}{1-{r^{*}}^{2}}}\;\{2-{r^{*}}^{2}-K(1-{r^{*}}^{2})^{2}\},
\end{align}
which is positive by the same argument as for the first derivative. Setting the numerator to zero and solving for $K$, we find that there are no solutions when $r$ is between zero and the appropriate root of a 9th order polynomial in $r$, which numerically is 0.946819. Between this value and 1 there are two solutions, for which the numerator is zero, given by the solutions to the two roots of a quartic polynomial in $K$. We can plot these solutions together with the upper and lower bounds for $K^{*}(r)$ and compare them with the true $K^{*}(r)$, calculated numerically, as shown in the right panel of Fig.~\ref{fig:kasr}.
\begin{figure}[H]
\centering
\includegraphics[scale=1.7]{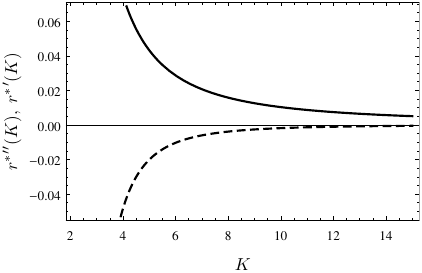}
\includegraphics[scale=1.7]{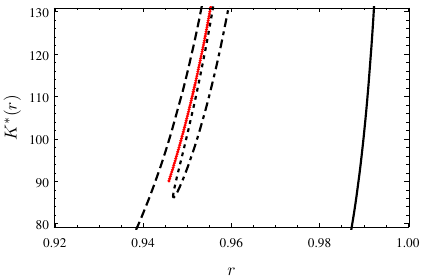}
\caption{Left: Interpolation of the first (solid) and second (dashed) derivatives of $r^{*}(K)$. Right: Comparison of the numerical solution for the bifurcation point $K^{*}(r)$ (red, dotted) with the upper bound $\frac{2-r^{2}}{(1-r^{2})^{2}}$ (long dashed) and the lower bound $\frac{2}{1-r^{2}}$ (solid), and with the solutions to the numerator of the second derivative being zero (short dashed and dash-dotted).}
\label{fig:kasr}
\end{figure}

The right panel of Fig.~\ref{fig:kasr} suggests that the second derivative also does not change sign.  Numerically solving for a pair $(K, r^{*}(K))$, and substituting this into the numerator, we see that the second derivative is negative. This is confirmed by the left panel of Fig.~\ref{fig:kasr}.
\end{proof}
\noindent To confirm the asymptotic solutions for $r^{*}(K)$ in properties 3 and 4, we plot them and compare them to the numerical solutions in Fig.~\ref{fig:rprop}. 

The next theorem gives the asymptotics of  $L^{*}(K)$ implicitly defined by \eqref{eq:criticalselfcon} in the limit as $K\rightarrow \infty$ and close to $(K, L)=(2, 0)$.

\begin{figure}[H]
\centering
\includegraphics[scale=1.7]{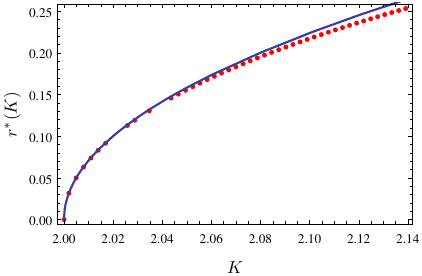}
\includegraphics[scale=1.7]{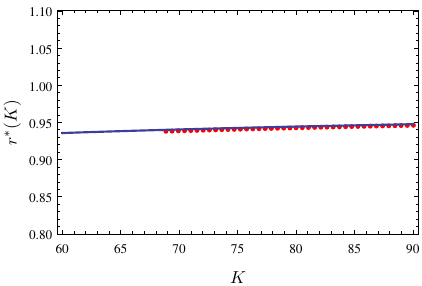}
\caption{Comparison of the numerical solution for the bifurcation point $r^{*}(K)$ with the asymptotic expressions for $r^{*}(K)$ given in properties 4 and 5 of Theorem \ref{lemma:rprop}, for $K$ close to 2 on the left and for $K$ large on the right.}
\label{fig:rprop}
\end{figure}

\begin{theorem}[\textbf{Asymptotic properties of the bifurcation line}]\ 
\label{thm:bifurcprop}
The derivative of $L^{*}(K)$, defined implicitly by \eqref{eq:criticalselfcon}, has the following properties:
\begin{enumerate}
\item $\lim_{K\rightarrow\infty}\frac{\partial L^{*}(K)}{\partial K} = -1$.
\item $\lim_{K\downarrow 2}\frac{\partial L^{*}(K)}{\partial K} = -\frac{1}{2}$.
\end{enumerate}
\end{theorem}
\begin{proof}
We begin by proving the existence of the limits, for which we need the following lemma.

\begin{lemma}[\textbf{Derivatives of $K\mapsto L^{*}(K)$}]\ 
\label{lemma:Lderivatives}
For all $K>2$,
\begin{enumerate}
\item $\frac{\partial L^{*}(K)}{\partial K} <0$.
\item $\frac{\partial^{2} L^{*}(K)}{\partial K^{2}} <0$.
\end{enumerate}
\end{lemma}

\begin{proof}
In order for $L^{*}(K)$ to be continuous by the IFT (in a similar way as in the proof of Theorem \ref{lemma:rprop}), we require that
\begin{align}
\label{eq:IFTlbound}
 L^{*}(K)> -\sqrt{K + K^{2} - \sqrt{K^{2}(1+4K)}}. 
\end{align}
We will see, numerically, that this bound is satisfied because it lies below another lower bound of $L^{*}(K)$. To rigorously show that this bound is satisfied, we expect that it is possible to use the same procedure as outlined for the bound on $r^{*}(K)$ in \eqref{eq:IFTrbound}. Now we start by differentiating \eqref{eq:criticalselfcon} with respect to $K$ and solving for $\partial_{K} L^{*}(K)$. This leads to
\begin{align}
\label{eq:Lfirstderivative}
\frac{\partial L^{*}(K)}{\partial K} = -\frac{(K-2)K^{3}+2K^{2}L^{*}(K)-2(K-1)KL^{*}(K)^{2}+2L^{*}(K)^{3}+L^{*}(K)^{4}}{(K-2)K^{3}-2(K-1)KL^{*}(K)^{2}+L^{*}(K)^{4}}.
\end{align}
Setting the numerator, which is a quartic polynomial in $L^{*}(K)$, equal to zero and solving for $L^{*}(K)$, we find one solution that lies above the critical condition for $L$ when fixing $K$, $-K+2$. The expression is too lengthy to present here and does not lead to any useful insight. Taking the derivative with respect to $K$ of \eqref{eq:Lfirstderivative}, substituting the expression for the first derivative \eqref{eq:Lfirstderivative} and setting the resulting numerator to zero, we are left with solving a 7th order polynomial for $L^{*}(K)$. Again the expression is lengthy and does not lead to any insight. Only one of the solutions to the 7th order polynomial lies above the critical line. Comparing these two solutions, one coming from the quartic polynomial and the other from the 7th order polynomial, we see numerically that the first is a lower bound for $L^{*}(K)$ and the second is an upper bound for $L^{*}(K)$, as seen in the right panel of Fig.~\ref{fig:lask}. This lower bound is an upper bound for the right-hand side of \eqref{eq:IFTlbound}, so that the conditions for the IFT are satisfied. The expression determining when the denominator of both the first and the second derivative is zero, obtained by setting their respective denominators to zero (which makes the derivatives diverge), is the same, and the only solution falling above the critical condition is upper bounded by the lower bound for $L^{*}(K)$ found above (as the solution to the quartic polynomial), so that the derivatives do not diverge. 

\begin{figure}[H]
\centering
\includegraphics[scale=1.7]{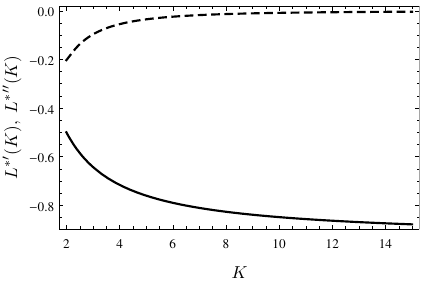}
\includegraphics[scale=1.7]{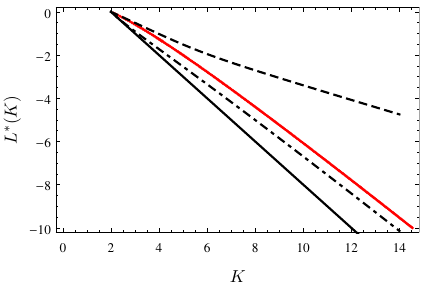}
\caption{Left: Interpolation of the first (solid) and second (dashed) derivatives of $L^{*}(K)$. Right: Comparison of the numerical solution for the bifurcation point $L^{*}(K)$ (red, dotted) with the upper bound/solution to the 7th order polynomial (dashed) and the lower bound/solution to the quartic polynomial (dot-dashed), as well as the critical condition for $L$ when fixing $K$, $-K+2$ (solid).}
\label{fig:lask}
\end{figure}

The right panel of Fig.~\ref{fig:lask} suggests that both the first derivative and the second derivative of $L^{*}(K)$ do not change sign as a function of $K$. Substituting a pair of values $K, L^{*}(K)$, solved for numerically, we confirm the statements in Lemma~\ref{lemma:Lderivatives}. This is also corroborated by the left panel of Fig.~\ref{fig:lask}.
\end{proof}

\begin{remark}
For the mathematical reader the numerical assistance in the argument above might not be satisfying. We suspect that the proof can be made rigorous by using the sequences of upper and lower bounds in \cite[Theorem 4]{S11} on $V$ in \eqref{eq:criticalselfcon}, in order to get upper and lower bounds for $L^{*}(K)$ that give a tighter wedge than the one in the right panel of Fig.~\ref{fig:lask}.  
\end{remark}

Due to Lemma \ref{lemma:Lderivatives} and the fact that $L^{*}(K)$ is bounded below by $-K+2$, we have that the limits exists.

We now turn to the proof of Theorem~\ref{thm:bifurcprop}. Abbreviate
\begin{align}
g(K, L) =  r^{*}(K, L) - V\big((K+L)r^{*}(K, L)\big).
\end{align}
By the implicit function theorem, we have 
\begin{align}
\frac{\partial L^{*}(K)}{\partial K} = - \frac{\partial_{K}g(K, L)}{\partial_{L}g(K, L)}.
\end{align}
Compute
\begin{align}
\label{eq:numerator}
\partial_{K}g(K, L) =& \;\frac{\frac{4K^{2}}{(K^{2}-L^{2})^{2}} - \frac{2}{K^{2}-L^{2}}}{2r^{*}(K, L)}\nonumber\\
&+ \Big((K + L)\frac{\frac{4K^{2}}{(K^{2}-L^{2})^{2}} - \frac{2}{K^{2}-L^{2}}}{2r^{*}(K, L)} +r^{*}(K, L) \,\Big)\\
&\times \Big(V^{2}\Big((K+L)r^{*}(K, L)\,\Big) - \tfrac12 - \tfrac12 S\Big((K+L)r^{*}(K, L)\,\Big)\Big)\nonumber
\end{align}
and 
\begin{align}
\label{eq:denominator}
\partial_{L}g(K, L) =& -\frac{2KL}{(K^{2}-L^{2})^{2}r^{*}(K, L)}\nonumber\\
&+ \Big(-(K+L)\frac{2K^{*}L}{(K^{2}-L^{2})^{2}r^{*}(K, L)} + r^{*}(K, L)\,\Big)\\
&\times \Big(V^{2}\Big((K+L)r^{*}(K, L)\,\Big) - \tfrac12 - \tfrac12 S\Big((K+L)r^{*}(K, L)\,\Big)\Big),\nonumber
\end{align}
where $S(x)= \frac{\text{I}_{2}(x)}{\text{I}_{0}(x)}$.

For property 1, we make the Ansatz $ L^{*}(K) = -aK + c, K\rightarrow\infty$ where $c=c(K)=o(K)$ (which is confirmed in Fig.~\ref{fig:asymphasediagram}). Taking the limit $K\to\infty$, we get zero for the first terms in the right-hand sides of \eqref{eq:numerator}--\eqref{eq:denominator}, i.e.,
\begin{align}
&\lim_{K\to\infty} \frac{\frac{4(K)^{2}}{(K^{2}-L^{2})^{2}} - \frac{2}{K^{2}-L^{2}}}{2\sqrt{1-\frac{2K}{K^{2}-L^{2}}}} = 0,\\
&\lim_{K\to\infty}\frac{2KL}{(K^{2}-L^{2})^{2}\sqrt{1-\frac{2K}{K^{2}-L^{2}}}} = 0,
\end{align}
where we have used the expression for $r^{*}(K, L)$ from \eqref{eq:criticalr}. The multiplication factors in the last line of the right-hand sides of \eqref{eq:numerator}--\eqref{eq:denominator} are the same, so we are left with calculating the limit as $K\rightarrow \infty$ of the quotient
\begin{align}
-\frac{\Bigg((K + (-aK + c))\frac{\frac{4K^{2}}{(K^{2}-(-aK + c)^{2})^{2}} - \frac{2}{K^{2}-(-aK + c)^{2}}}{2r^{*}(K, -aK + c)} + r^{*}(K, -aK+c) \Bigg)}{\Bigg(-(K+(-aK + c))\frac{2K(-aK + c)}{(K^{2}-(-aK + c)^{2})^{2}r^{*}(K, -aK+c)} + r^{*}(K, -aK+c)\Bigg)},\quad c=o(K).
\end{align}
A straightforward but tedious calculation (with the help of MATHEMATICA) shows that this limit is $-1$.

For property 2, we must find the limit of $- \frac{\partial_{K}g(K, L)}{\partial_{L}g(K, L)}$ as we approach the point $(K, L)=(2, 0)$ along the line $L^{*}(K)$. We make the Ansatz $L^{*}(K) =(K-2)b +o(1)$, $K\downarrow 2$. Making this replacement in the expression for the derivative and doing a Taylor expansion around $K=2$, we obtain after a tedious calculation (with the help of MATHEMATICA),
\begin{align}
\label{eq:asymnumerator}
\lim_{K\downarrow 2}\partial_{K}g(K, L)|_{L = (K-2)b} = -\sqrt{K-2}\Big(\frac{3}{8\sqrt{2}} + \frac{b}{4\sqrt{2}}\Big)
\end{align}
for the terms in the numerator and
\begin{align}
\label{eq:asymdenominator}
\lim_{K\downarrow 2}\partial_{L}g(K, L)|_{L = (K-2)b} = -\frac{\sqrt{K-2}}{2\sqrt{2}}
\end{align}
for the terms in the denominator.
Combining \eqref{eq:asymnumerator}--\eqref{eq:asymdenominator} we obtain 
\begin{align}
b = -\frac{1}{4}(3+2b),
\end{align}
so that $b=-\frac{1}{2}$.
\end{proof}
Properties 1 and 2 are confirmed by the left panel of Fig~\ref{fig:lask}. It seems that $K\mapsto L^{*}(K)$ for large $K$ does not have an asymptote, since when we take the limit after the replacement $L^{*}(K) = -K + c$ we get an equation for the bifurcation point that reads
\begin{align}
\label{eq:asympc}
\sqrt{1 - \frac{1}{c}} = V\Bigg(c\sqrt{1 - \frac{1}{c}}\;\Bigg).
\end{align}
The only solution to this equation is $c=1$, which is not possible because it would place the asymptote below the critical line. This suggests that $c=c(K)$ grows as a function of $K$, but that this growth is sublinear. 

\section{Simulation}
\label{sec:simulation}
Fixing the phase difference is not physical, since the system will relax into a steady state and will choose the angles that are the least costly energetically. Studying the dynamics of the transitions between states or the stability properties of the possible states are both difficult tasks. However, we expect that the non-symmetric state is either unstable or metastable and using simulations we can observe what type of transitions one might expect between the possible states. To see this, we take the initial distribution for both populations to have mean $\pi$, but choose the second community to have a slightly larger variance initially, meaning that the synchronization level starts lower. The outcome of the simulation can be seen in Fig \ref{fig:metastable}. It seems that the community with less synchronization initially is suppressed by the community with more synchronization, until the `push' from the latter becomes too strong. This is reflected in the angles, which stay relatively close together for a while, before moving apart. This type of transition seems only to occur when the parameters are chosen such that the non-symmetric solutions discussed above exist.

We expect that the most stable state is the symmetric solution with the largest synchronization level (i.e., the largest effective interaction strength). For example, if $K=5$ and $L=2$, then the symmetric solution with phase difference $\psi=\pi$ has $r_{1}=r_{2}=0.724\ldots$, while the symmetric solution with phase difference $\psi=0$ has $r_{1}=r_{2}=0.918\ldots$. The first state is unstable/metastable, the second state is stable. The transition from the one to the other is shown in Fig~\ref{fig:stabletrans}. 
\begin{figure}[H]
\includegraphics[scale=1.7]{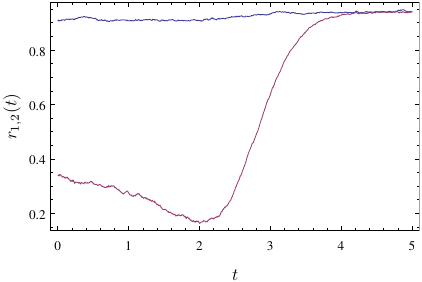}
\includegraphics[scale=1.7]{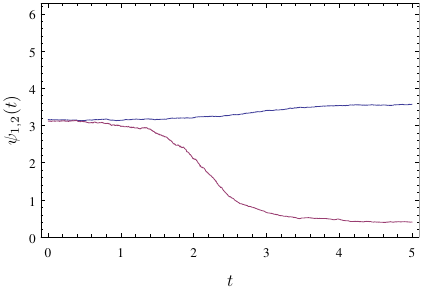}
\caption{Simulation of $1000$ oscillators per community with $K=7$ and $L=-2$. The time step is set at $\ddd t=0.01$. The left image shows the synchronization levels, the right image the phase averages.}
\label{fig:metastable}
\end{figure}
\vspace{-0.8cm}
\begin{figure}[H]
\includegraphics[scale=1.7]{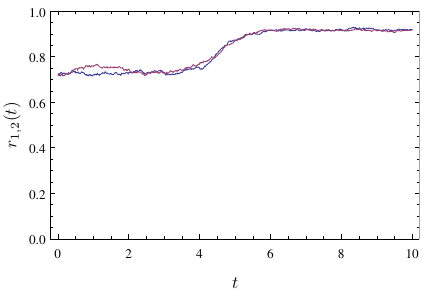}
\includegraphics[scale=1.7]{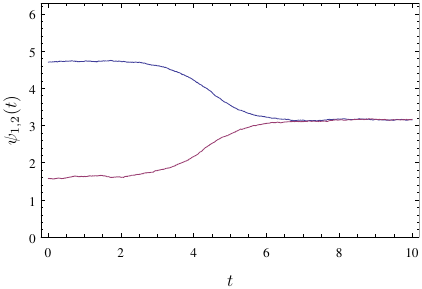}
\caption{Simulation of $1000$ oscillators per community with $K=5$ and $L=2$. The time step is set at $\ddd t=0.01$. The left image shows the synchronization levels, the right images the phase averages.}
\label{fig:stabletrans}
\end{figure}

To be clear, these simulations are not meant to prove any stability properties or transitions, but are useful for determining what types of stability properties or transitions we may expect. They also suggest that much interesting work remains to be done.

\appendix

\section{Concavity of V}
\label{app:concavity}
Recall that
\begin{equation}
\label{eq:psi}
V(x) = \frac{\int_{0}^{2\pi}\eee^{x\cos\theta}\cos \theta\, \ddd \theta}{\int_{0}^{2\pi}\eee^{x\cos\theta}\, \ddd \theta}.
\end{equation}
The first derivative of \eqref{eq:psi} is 
\begin{equation}
\label{eq:psiprime}
\partial_{x}V(x) = \frac{\int_{0}^{2\pi}\eee^{x\cos\theta}\, \ddd \theta\int_{0}^{2\pi}\eee^{x\cos\theta}\cos^{2}\theta\, \ddd \theta - \big(\int_{0}^{2\pi}\eee^{x\cos\theta}\cos\theta\, \ddd \theta\big)^{2}}{(\int_{0}^{2\pi}\eee^{x\cos\theta}\, \ddd \theta)^{2}}.
\end{equation}
We can rewrite
\begin{align}
\int_{0}^{2\pi}\eee^{x\cos\theta}\cos\theta\, \ddd \theta &= \int_{0}^{2\pi}\eee^{\frac{1}{2}x\cos\theta}\eee^{\frac{1}{2}x\cos\theta}\cos\theta\, \ddd \theta\\
&\leq \Big(\int_{0}^{2\pi}\eee^{x\cos\theta}\, \ddd \theta\Big)^{1/2}\Big(\int_{0}^{2\pi}\eee^{x\cos\theta}\cos^{2}\theta\, \ddd \theta\Big)^{1/2},\nonumber
\end{align}
where we have used Holder's inequality in the second line. Taking the square on both sides, we obtain
\begin{equation}
\Big(\int_{0}^{2\pi}\eee^{x\cos\theta}\cos\theta\, \ddd \theta\Big)^{2} \leq \int_{0}^{2\pi}\eee^{x\cos\theta}\, \ddd \theta\int_{0}^{2\pi}\eee^{x\cos\theta}\cos^{2}\theta\, \ddd \theta,
\end{equation}
which proves that \eqref{eq:psiprime} is non-negative. We evaluate \eqref{eq:psiprime} at $x=0$, to get
\begin{equation}
\partial_{x}V(x)|_{x=0} = \frac{2\pi\times\pi}{(2\pi)^{2}} = \frac{1}{2}. 
\end{equation}
For the second derivative we rewrite 
\begin{equation}
\partial_{x}V(x) = \text{I} - \text{II} = \frac{\int f''(x, \theta)\,\ddd \theta}{\int f(x, \theta)\,\ddd \theta} - \frac{(\int f'(x, \theta)\,\ddd \theta)^{2}}{(\int f(x, \theta)\,\ddd \theta)^{2}},
\end{equation}
where $f(x, \theta) = \eee^{x\cos\theta}$ and the prime refers to the derivative with respect to $x$. The integrals are always from $0$ to $2\pi$. Taking the derivative of the first term, we find
\begin{equation}
\text{I}' = \frac{\int f(x, \theta)\,\ddd \theta\int f'''(x, \theta)\,\ddd \theta - \int f''(x, \theta)\,\ddd \theta\int f'(x, \theta)\,\ddd \theta}{(\int f(x, \theta)\,\ddd \theta)^{2}},
\end{equation}
while for the second we find
\begin{equation}
\text{II}' = 2\frac{\int f'(x, \theta)\,\ddd \theta}{\int f(x, \theta)\,\ddd \theta}\times\frac{\int f(x, \theta)\,\ddd \theta\int f''(x, \theta)\,\ddd \theta - (\int f'(x, \theta)\,\ddd \theta)^{2}}{(\int f(x, \theta)\,\ddd \theta)^{2}}.
\end{equation}
Using a common denominator, we can write the difference as 
\begin{align}
\partial_{x}^{2}V(x) &= \frac{1}{\big(\int f(x, \theta)\,\ddd \theta\big)^{3}}\Big[\Big(\int f(x, \theta)\,\ddd \theta\Big)^{2}\int f'''(x, \theta)\,\ddd \theta \nonumber\\
&- 3\int f(x, \theta)\,\ddd \theta\int f'(x, \theta)\,\ddd \theta\int f''(x, \theta)\,\ddd \theta + 2 \Big(\int f'(x, \theta)\,\ddd \theta\Big)^{3} \Big]
\end{align}
To continue, we first let
$$2c =  \int e^{x \cos \theta} \ddd \theta$$
making the desired expression into
\begin{equation}
V''(x) = \Big[\int f'''(x, \theta)\,\frac{\ddd  \theta}{2c} 
- 3\int f'(x, \theta)\,\frac{\ddd  \theta}{2c}\int f''(x, \theta)\,\frac{\ddd  \theta}{2c} + 2 \Big(\int f'(x, \theta)\,\frac{\ddd  \theta}{2c}\Big)^{3} \Big].
\end{equation}
With the functions 
$$\arccos_1: (-1,1) \to (0,\pi),\quad
\arccos_2: (-1,1) \to (\pi,2\pi),$$
we can perform the change of variable $u=\cos\theta$, i.e., $\theta = \arccos u$ and 
\begin{equation}
\ddd \theta = \frac{-\ddd u}{\sqrt{1-u^{2}}}.
\end{equation}
Here we get
\begin{align}
\int_0^{2\pi} (\cos\theta )^k  e^{x \cos \theta} \frac{\ddd  \theta}{2c} &=
\int_0^{\pi} (\cos\theta )^k  e^{x \cos \theta} \frac{\ddd  \theta}{2c} +
\int_\pi^{2\pi} (\cos\theta )^k  e^{x \cos \theta} \frac{\ddd  \theta}{2c}\\
&=-\int_{1}^{-1} u^k e^{x u} \frac{\ddd u}{c\sqrt{1 - u^2}}\nonumber\\
 &= \int_{-1}^1 u^k e^{x u} \frac{\ddd u}{c\sqrt{1 - u^2}}.\nonumber
\end{align}
where we have used $\arccos_1$ for the first integral and $\arccos_2$ for the second. Note that, when $k=0$,
\begin{equation}
\label{eq:kiszero}
\int_{-1}^1  e^{x u} \frac{\ddd u}{c\sqrt{1 - u^2}} = 1.
\end{equation}
With the change of measure
\begin{align}
\ddd\nu(u) = \frac{1}{c\sqrt{1-u^2}}\ddd u
\end{align}
we obtain
\begin{equation}
V''(x) = \Big[\int_{-1}^{1} u^3e^{xu}\,\ddd  \nu(u) 
- 3\int_{-1}^{1} ue^{xu}\,\ddd \nu(u)\int_{-1}^{1} u^2e^{xu}\,\ddd  \nu(u) + 2 \Big(\int_{-1}^{1} ue^{xu}\,\ddd  \nu(u)\Big)^{3} \Big]
\end{equation}
and defining $m = m(x) = \int_{-1}^{1} ue^{xu}\,\ddd  \nu(u)$ we get 
\begin{equation}
\label{eq:star}
V''(x) = \Big[\int_{-1}^{1} u^3e^{xu}\,\ddd \nu(u) 
- 3m\int_{-1}^{1} u^2e^{xu}\,\ddd \nu(u) + 2 m^{3} \Big].
\end{equation}
Note that, due to \eqref{eq:kiszero}, and since $\int_{-1}^{1} 3m^2 \,u\, e^{xu}\,\ddd \nu(u) = 3m^3$, \eqref{eq:star} equals
\begin{equation}
\int_{-1}^{1} (u-m)^3e^{xu}\,d \nu(u).
\end{equation}
We can check this by writing
\begin{align}
\int_{-1}^{1} (u-m)^3e^{xu}\,\ddd \nu(u) &= \int_{-1}^{1} (u^3-3mu^2 + 3m^2 u - m^3)e^{xu}\,\ddd \nu(u)\\
&= \Big[\int_{-1}^{1} u^3e^{xu}\,\ddd \nu(u) 
- 3m\int_{-1}^{1} u^2e^{xu}\,\ddd \nu(u) + 2 m^{3} \Big]. 
\end{align}
To complete the proof we state \cite[Lemma 4]{PP81}, suitably adapted. 
\begin{lemma}
Let $\nu$ be an even probability measure with support on $[-1, 1]$, and suppose that $\nu$ is absolutely continuous, i.e., $\ddd \nu(\sigma) = f(\sigma)\ddd \sigma$, with $f$ non-decreasing on $[0, 1]$. Then $\nu \in \mathscr{P}$, the class of all probability measures on $\mathbb{R}$ with compact support, is such that 
\begin{equation}
\int_{-1}^{1} \eee^{k\sigma}(m-\sigma)^{p} \ddd \nu (\sigma)\geq 0,
\end{equation}
where 
\begin{equation}
m = m(k) = \frac{\int \sigma \eee^{k\sigma}\ddd \nu(\sigma)}{\int \eee^{k\sigma}\ddd\nu(\sigma)},
\end{equation}
and $k = Jm + h$, where $J$ is the mean-field interaction strength and $h$ is the magnetic field strength of the spin system.
\end{lemma}
With the identification
\begin{equation}
\sigma = u, \quad k=x, \quad p=3
\end{equation}
and taking out a negative, we complete the proof. To get the strict inequality we note that the equality in the lemma does not hold for our choice of $\nu(u)$.

%%%%%%%%%%% REFERENCES %%%%%%%%%%%%%%%%%%%%%

\end{document}